\documentclass[11pt,twoside]{article}
\usepackage[pdftex]{graphicx}
\usepackage{amssymb}
\usepackage{amsmath}
\usepackage{amsthm}
\usepackage{times,cite,bm}
\usepackage{color}
\usepackage{float}
\usepackage[leftcaption]{sidecap}

\newcommand{\be}{\begin{eqnarray}}
\newcommand{\ee}{\end{eqnarray}}
\newcommand{\bez}{\begin{eqnarray*}}
\newcommand{\eez}{\end{eqnarray*}}
\theoremstyle{plain}
\newtheorem{thm}{Theorem}[section]

\newtheorem{proposition}[thm]{Proposition}

\newtheorem{lemma}[thm]{Lemma}
\newtheorem{corollary}[thm]{Corollary}

\theoremstyle{definition}

\newtheorem{definition}[thm]{Definition}
\theoremstyle{remark}

\newtheorem{remark}[thm]{Remark}
\newtheorem{example}[thm]{Example}

\title{KP solitons, higher Bruhat and Tamari orders} 

\author{ {\scshape Aristophanes Dimakis} \\
 Department of Financial and Management Engineering, \\
 University of the Aegean, 41, Kountourioti Str., GR-82100 Chios, Greece \\
 E-mail: \emph{dimakis@aegean.gr}
             \and
 {\scshape Folkert M\"uller-Hoissen} \\
 Max-Planck-Institute for Dynamics and Self-Organization \\
 Bunsenstrasse 10, D-37073 G\"ottingen, Germany \\
 E-mail: \emph{folkert.mueller-hoissen@ds.mpg.de} 
       }

\date{ }

\begin{document}

\maketitle

\begin{abstract}
In a tropical approximation, any tree-shaped line soliton solution, a member of  
the simplest class of soliton solutions of the Kadomtsev-Petviashvili (KP-II) equation, 
determines a chain of planar rooted binary trees, connected by right rotation. More precisely, 
it determines a maximal chain of a Tamari lattice. 
We show that an analysis of these solutions naturally involves 
higher Bruhat\index{order!higher Bruhat} and higher Tamari orders\index{order!higher Tamari}.
\end{abstract}

\section{Introduction}
\setcounter{equation}{0}
Waves on a fluid surface show a very complex behavior in general. Only under special 
circumstances can we expect to observe a more regular pattern. For shallow water waves,  
the Kadomtsev-Petviashvili (KP) equation\index{KP!equation} 
\bez
      (-4 \, u_t + u_{xxx} + 6 \, u u_x )_x + 3 \, u_{yy} = 0
\eez
(where e.g. $u_t = \partial u/\partial t$) 
provides an approximation under the conditions 
that the wave dominantly travels in the $x$-direction, the wave length is long as 
compared with the water depth, and the effect of the nonlinearity is about the same order 
as that of dispersion.\footnote{The physical form of this equation 
is obtained by suitable rescalings of $x,y,t$ and $u$, involving physical parameters. } 
More precisely, this is the KP-II equation\index{KP!equation}, but we will write KP, for short. 
It generalizes the famous Korteweg-deVries (KdV) equation, 
which describes waves moving in only one spatial dimension. 
Although the KdV equation is much better established as an approximation of the more  
general water wave equations, recent studies also confirm the physical relevance of 
KP\index{KP!equation} \cite{DM_Kodama10}. 

In \cite{DM_DMH11KPT} we studied the soliton solutions of the KP equation\index{KP!equation} in a 
\emph{tropical}\index{tropical} approximation, which reduces them to networks 
formed by line segments in the $xy$-plane, evolving in time $t$. 
A subclass corresponds to evolutions (in time $t$) in the set of (planar) rooted binary trees\index{tree!planar rooted binary}. 
At transition events, the binary tree type changes (through a tree that 
is not binary). 
It turned out that the time evolution is simply given by right rotation in a tree, 
and the solution evolves according to a maximal chain of a 
\emph{Tamari lattice}\index{Tamari!lattice}
\cite{DM_Tamari1951thesis,DM_Tamari62}, see Figure~\ref{DM_fig:T4}. 
\begin{figure}[H] 
\begin{center}
\includegraphics[scale=.9]{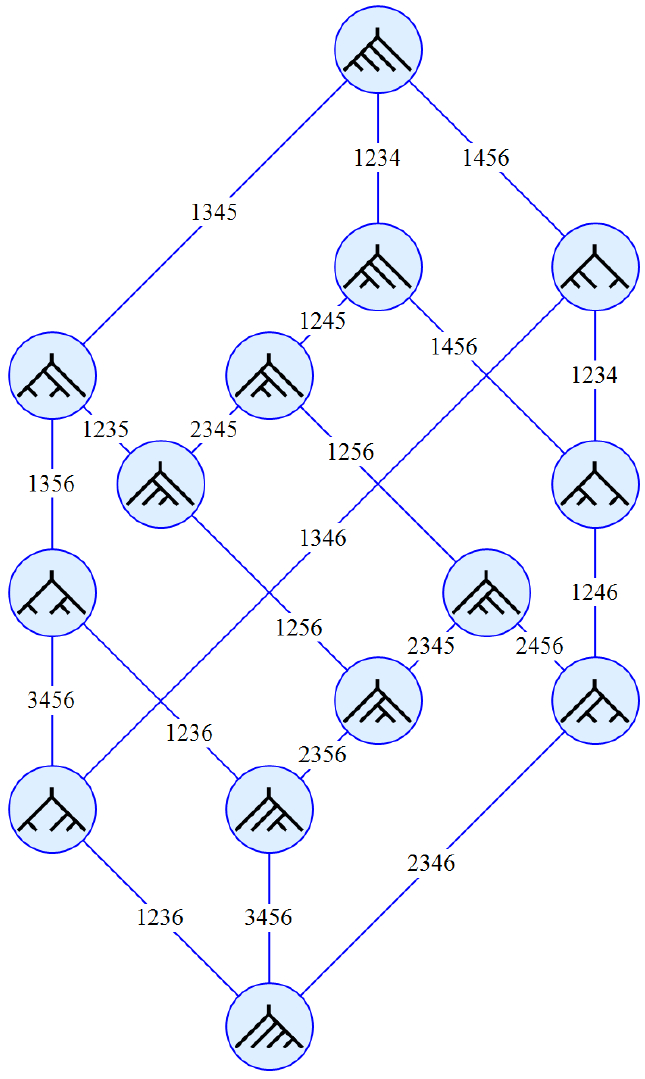} 
\parbox{15cm}{
\caption{The Tamari lattice\index{Tamari!lattice} $\mathbb{T}_4$ in terms of rooted binary trees. 
The top node shows a left comb tree that represents the structure of a certain KP line 
soliton\index{KP!soliton} family (with six asymptotic branches in the $xy$-plane) as $t \to -\infty$. 
A label $ijkl$ assigned to an edge indicates the transition time $t_{ijkl}$ at which the soliton 
graph changes its tree type via a `rotation' (see Section~\ref{DM_sec:tropical}). 
The values of the parameters, on which the solutions depend, determine the linear order of 
the `critical times' $t_{ijkl}$, and thus decide which chain is realized. 
The bottom node shows a right comb tree, which represents the tree type of the soliton as 
$t \to \infty$. The special family of solutions thus splits into classes corresponding to 
the maximal chains of $\mathbb{T}_4$. For each Tamari lattice\index{Tamari!lattice}, 
there is a family of KP line solitons\index{KP!soliton} that realizes its maximal chains in this way. \label{DM_fig:T4} }
}
\end{center} 
\end{figure} 

In this realization of Tamari lattices\index{Tamari!lattice}, the underlying set consists of 
states of a physical system, here the tree-types of a soliton configuration in the $xy$-plane. 
The Tamari poset\index{Tamari!lattice} (partially ordered set) structure describes the possible ways in 
which these states are allowed 
to evolve in time, starting from an initial state (the top node) and ending in a 
final state (the bottom node). 

Figure~\ref{DM_fig:rotation} displays a solution evolving via 
a tree rotation, and further provides an idea how this can be understood 
in terms of an arrangement of planes in three-dimensional space-time (after idealizing 
line soliton branches to lines in the $xy$-plane, see Section~\ref{DM_sec:tropical}). 

\begin{figure}
\begin{center}
\includegraphics[scale=.5]{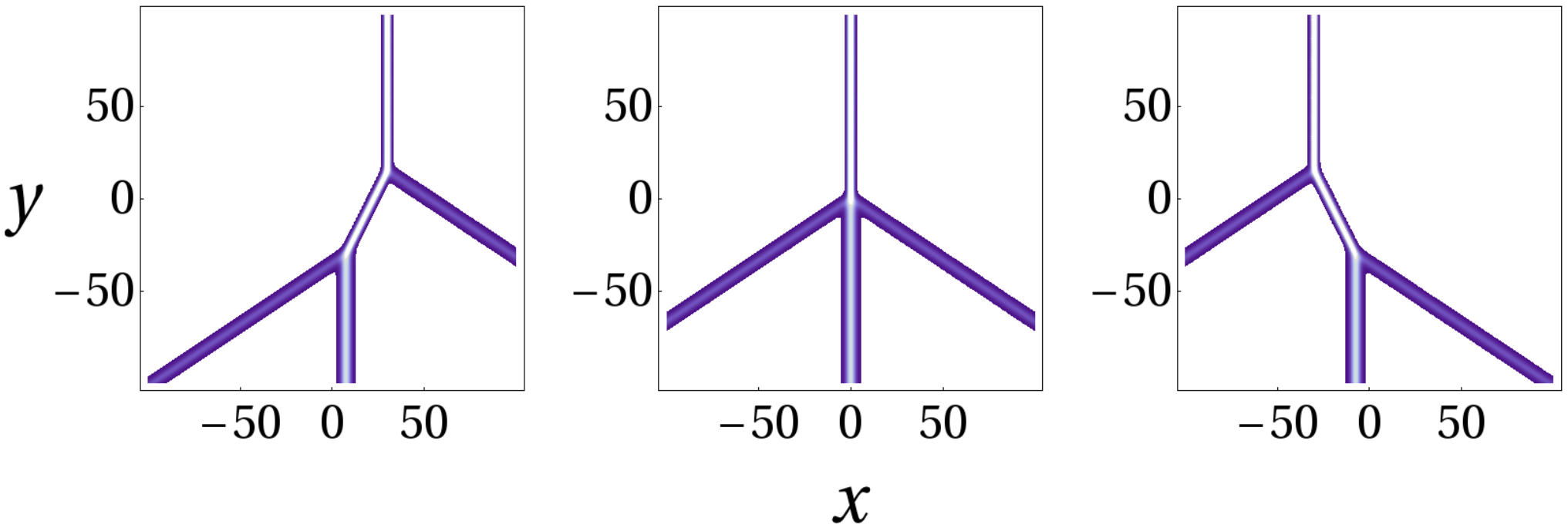} 
\hspace{.6cm}
\includegraphics[scale=.42]{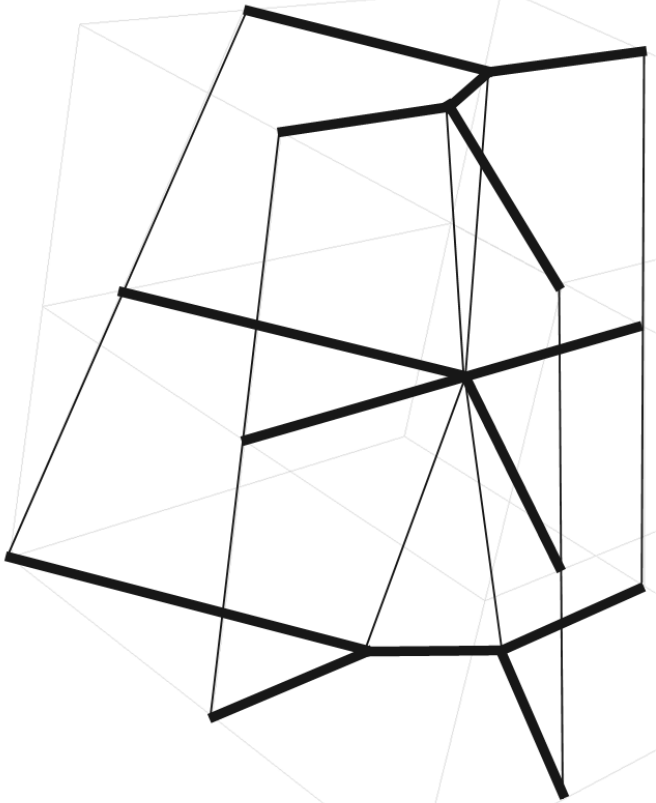}  
\parbox{15cm}{ 
\caption{Density plots of a line soliton solution at three successive times, exhibiting 
a tree rotation. To the right is a corresponding space-time view in terms of intersecting planes 
(here time flows upward). \label{DM_fig:rotation} }
}
\end{center} 
\end{figure} 

In this work we show that the classification of possible evolutions of tree-shaped 
KP line solitons\index{KP!soliton}  
involves \emph{higher Bruhat orders}\index{order!higher Bruhat} 
 \cite{DM_Manin+Schecht86a,DM_Manin+Shekhtman86b,DM_Manin+Schecht89,DM_Ziegler93}. 
Moreover, we are led to associate with each higher Bruhat order a 
\emph{higher Tamari order}\index{order!higher Tamari} 
via a surjection, in a way different from what has been considered previously. There is some 
evidence that our higher Tamari orders coincide with `higher Stasheff-Tamari posets' 
introduced by Kapranov and Voevodsky \cite{DM_Kapr+Voev91} (also see \cite{DM_Edel+Rein96}), 
but a closer comparison will not be undertaken in this work. 

In Section~\ref{DM_sec:KP_line_solitons}, we briefly describe the general class of 
KP soliton\index{KP!soliton} 
solutions. In Section~\ref{DM_sec:tropical}, we concentrate on the abovementioned subclass 
of tree-shaped solutions, in the tropical\index{tropical} approximation, and somewhat improve 
results in \cite{DM_DMH11KPT}.
Section~\ref{DM_sec:hBT} recalls results about higher Bruhat orders\index{order!higher Bruhat} and 
extracts from the analysis of tree-shaped KP line solitons\index{KP!soliton} a reduction to 
higher Tamari orders\index{order!higher Tamari}. 
Section~\ref{DM_sec:polygonal_rels} proposes a hierarchy of monoids that
expresses the hierarchical structure present in the KP soliton\index{KP!soliton} problem. 
This makes contact with \emph{simplex equations}\index{simplex equation}
\cite{DM_Zamolodchikov81,DM_Bazh+Stro82,DM_Frenkel+Moore91,DM_Lawrence97} 
and provides us with an algebraic method 
to construct higher Bruhat and higher Tamari orders. 
Section~\ref{DM_sec:remarks} contains some additional remarks. 
Throughout this work, Hasse diagrams\index{Hasse diagram} of posets will be 
displayed upside down (i.e. with the lowest element(s) at the top).

\section{KP solitons}
\label{DM_sec:KP_line_solitons}
The line soliton solutions of the KP-II equation\index{KP!equation} are parametrized by 
the totally non-negative 
Grassmannians $\mathrm{Gr}^{\geq}_{n,M+1}$ \cite{DM_Kodama10}, which is easily recognized in 
the Wronskian form of the solutions. Translating the KP equation\index{KP!equation} via
\bez
         u = 2 \, \log(\tau)_{xx} \, ,
\eez
into a bilinear equation in the variable $\tau$, these solutions are given by 
\bez
      \tau = f_1 \wedge f_2 \wedge \cdots \wedge f_n \, ,
\eez
where
\bez
    f_i = \sum_{j=1}^{M+1} a_{ij} \; e_j \, , \qquad
    e_j = e^{\theta_j} \, , \qquad
    \theta_j = \sum_{r = 1}^M p_j^r \; t^{(r)} + c_j  \; .
\eez
Here $t^{(r)}$, $r=1,\ldots,M$, are independent real variables that may be regarded as coordinates 
on $\mathbb{R}^M$, and we set $t^{(1)}=x$, $t^{(2)}=y$, $t^{(3)}=t$. The variables $t^{(r)}$, $r>3$, 
are the additional evolution variables that appear in the \emph{KP hierarchy}\index{KP!hierarchy}, which extends 
the KP equation\index{KP!equation} to an infinite set of compatible PDEs. 
Furthermore, $p_j,c_j,a_{ij}$ are real constants, and without restriction of generality 
we can and will assume that
\bez
      p_1 < p_2 < \cdots < p_{M+1} \; .  
\eez
The exterior product on the space of functions generated by the exponential functions 
$e_j$, $j=1,\ldots,M+1$, is defined by
\bez
    e_{i_1} \wedge \cdots \wedge e_{i_m} 
 = \Delta(p_{i_1},\ldots,p_{i_m}) \, e_{i_1} \cdots e_{i_m} \, ,
\eez
with the Vandermonde determinant\index{Vandermonde determinant} 
\bez
     \Delta(p_{i_1},\ldots,p_{i_m}) 
   = \left| \begin{array}{cccc} 1 & p_{i_1} & \cdots & p_{i_1}^{m-1} \\
                   1 & p_{i_2} & \cdots & p_{i_2}^{m-1}  \\
                   \vdots & \vdots & \cdots & \vdots \\
                   1 & p_{i_m} & \cdots & p_{i_m}^{m-1} \end{array} \right|
   = \prod_{1 \leq r < s \leq m} ( p_{i_s} - p_{i_r} )  \; .  
\eez
Now we can express $\tau$ as
\bez
    \tau = \sum_{I \in {[M+1] \choose n}} A_I \, \Delta(p_I) \, e_I \, , 
\eez
where $[m] = \{1,2,\ldots,m\}$, and ${[m] \choose n}$ denotes the set of $n$-element 
subsets of $[m]$. 
Numbering the elements of a subset $I=\{i_1,\ldots,i_n\}$ such that  
$i_1 < \cdots < i_n$, we set $\Delta(p_I) = \Delta(p_{i_1},\ldots,p_{i_n})$. 
Finally, $A_I$ denotes the maximal minor with columns $i_1, \ldots, i_n$ of the matrix
\bez
    A = \left( \begin{array}{ccc} a_{1,1} & \cdots & a_{1,M+1} \\
                \vdots & \ddots & \vdots \\
               a_{n,1} & \cdots & a_{n,M+1} 
               \end{array} \right) \; .
\eez 
For regular (soliton) solutions, the Pl\"ucker coordinates\index{Pl\"ucker coordinates}
$A_I$ have to be 
\emph{non-negative} real numbers (and at least one has to be different from zero). 
In the following we concentrate on the subclass of solutions parametrized by 
$\mathrm{Gr}^{\geq}_{1,M+1}$, i.e.
\be
    \tau = e_1 + \cdots + e_{M+1} \, ,  \label{DM_tau_simple_class}
\ee
where we absorbed the positive constants $a_{1,j}$ into the constants $c_j$. 
To good approximation, a \emph{general} line soliton solution can be understood as a 
\emph{superimposition} of solutions from the subclass (see \cite{DM_DMH11KPT}).

\section{Tropical approximation of a subclass of KP line solitons}
\label{DM_sec:tropical}
Let us fix $M \in \mathbb{N}$, constants $p_1 < p_2 < \cdots < p_{M+1}$ 
and $c_i$, $i=1,\ldots,M+1$. 
The behavior of $\tau \, : \, \mathbb{R}^M \to \mathbb{R}$, given by (\ref{DM_tau_simple_class}), 
is best understood in a tropical\index{tropical} approximation of $\log(\tau)$. 
In a region where some phase, say $\theta_i$, \emph{dominates} all others, i.e. 
$\theta_i > \theta_j$ for all $j \neq i$, we have
\bez
    \log(\tau) = \theta_i + \log\Big( 1 + \sum_{j=1 \atop j \neq i}^{M+1} e^{-(\theta_i-\theta_j)} \Big) 
            \simeq \theta_i \; .
\eez
As a consequence, 
\bez
    \log(\tau) \simeq \max\{\theta_1, \ldots, \theta_{M+1} \} \, ,
\eez
where the right hand side can be regarded as a \emph{tropical}\index{tropical} version of $\log(\tau)$. 
Sufficiently away from the boundary of a dominating-phase region, $\log(\tau)$ is linear 
in $x$, so that $u$ vanishes. A crucial observation is that a line soliton branch in the $xy$-plane, 
for fixed $t^{(r)}$, $r>2$, corresponds to a boundary line between two dominating-phase 
regions. Viewing it in space-time, by regarding $t$ as a coordinate of an additional dimension, 
or more generally 
in the extended space $\mathbb{R}^M$ by adding dimensions corresponding to the evolution variables 
$t^{(r)}$, $r=3,\ldots,M$, the boundary consists piecewise of affine hyperplanes. 
Let $\mathcal{U}_i$ denote the region where $\theta_i$ is not dominated by any other phase, i.e. 
\be
    \mathcal{U}_i 
 = \{ \mathbf{t} \in \mathbb{R}^M \, | \, \max\{\theta_1, \ldots, \theta_{M+1} \} 
     = \theta_i \} 
 = \bigcap_{k \neq i}  \{ \mathbf{t} \in \mathbb{R}^M \, | \, \theta_k \leq \theta_i \}  
    \; . \label{DM_cU_i}
\ee
In the tropical\index{tropical} approximation, a description of KP line solitons\index{KP!soliton}
 amounts to an ana\-lysis of intersections of such regions, i.e.
\bez
    \mathcal{U}_I = \mathcal{U}_{i_1} \cap \cdots \cap \mathcal{U}_{i_n} \, , \quad
    I = \{i_1,\ldots,i_n\} \in {\Omega \choose n} \, , \quad  
    \Omega := [M+1] = \{1,\ldots,M+1\} \; .
\eez
This is a subset of the affine space 
\bez
  \mathcal{P}_I = \{ \mathbf{t} \in \mathbb{R}^M \, | \, \theta_{i_1} 
                        = \cdots = \theta_{i_n} \} 
             \qquad \qquad n >1 \, ,
\eez
which is easy to deal with (see below). It is more difficult to determine which parts of 
$\mathcal{P}_I$ are \emph{visible}, i.e. belong to $\mathcal{U}_I$.
Fixing the values of $t^{(r)}$, $r>2$, determines 
a line soliton segment in the $xy$-plane if $n=2$, and a meeting point of $n$ such segments 
if $n>2$.\footnote{For generic values of $t^{(r)}$, $r>2$, we see line soliton segments 
and meeting points of \emph{three} segments in the $xy$-plane. Meeting points of more than 
three segments only occur for special values. } 
In order to decide about visibility, i.e. whether a point of $\mathcal{P}_I$ lies 
in $\mathcal{U}_I$, a formula is needed to compare the values 
of all phases at this point, see (\ref{DM_theta-diff-on-Pred}) below.

Let us first look at $\mathcal{P}_I$ in more detail. Introducing a real \emph{auxiliary variable} 
$t^{(0)}$, the equation $\theta_{i_1} = \cdots = \theta_{i_n} = -t^{(0)}$ results in the 
linear system\footnote{We note that the full set of equations
$t^{(0)} + p_i \, t^{(1)} + p_i^2 \, t^{(2)} + \cdots + p_i^M \, t^{(M)} = -c_i$, 
$i=1,\ldots, M+1$, defines a \emph{cyclic hyperplane arrangement} \cite{DM_Ziegler93} in $\mathbb{R}^{M+1}$ 
with coordinates $t^{(0)}, \ldots, t^{(M)}$. }
\bez
   t^{(0)} + p_{i_j} \, t^{(1)} + p_{i_j}^2 \, t^{(2)} + \cdots + p_{i_j}^{n-1} \, t^{(n-1)}
   = - \tilde{c}_{i_j}  \, , \qquad
   j=1,\ldots, n \, ,
\eez
where
\bez
    \tilde{c}_{i_j} = c_{i_j} + p_{i_j}^{n} \, t^{(n)} + \cdots + p_{i_j}^{M} \, t^{(M)} \; .
\eez
This fixes the first $n-1$ coordinates as linear functions of the remaining coordinates, 
$t_I^{(k)} = t_I^{(k)}(t^{(n)},\ldots,t^{(M)})$, $k=1,\ldots,n-1$. In particular, we obtain
(also see \cite{DM_DMH11KPT}, Appendix~A)
\be
    t_I^{(n-1)} &=& - \sum_{r=1}^{M+1-n} h_r(p_I) \, t^{(n+r-1)} - c_I  \nonumber \\
                &=& - (p_{i_1} + \cdots + p_{i_n}) \, t^{(n)} - \tilde{c}_I \, ,
                   \label{DM_t_I^n-1__h_r}
\ee
where $h_r(p_I) = h_r(p_{i_1},\ldots,p_{i_n})$ is the $r$-th \emph{complete symmetric polynomial} 
\cite{DM_Macd95} in the variables $p_{i_1},\ldots,p_{i_n}$, and
\bez
  c_I = \frac{1}{\Delta(p_I)} \sum_{s=1}^n (-1)^{n-s} \, c_{i_s} \, \Delta(p_{I\setminus \{i_s\}})
          \, , \qquad
  \tilde{c}_I = \sum_{r=2}^{M+1-n} h_r(p_I) \, t^{(n+r-1)} + c_I \; .
\eez 
We note that $\tilde{c}_I$ depends on $t^{(n+1)},\ldots, t^{(M)}$. 
Here are some immediate consequences:
\begin{itemize}
\item Since obviously $\mathcal{P}_I \subset \mathcal{P}_J$ for $J \subset I$, on $\mathcal{P}_I$ we have
\bez
    t_I^{(n-2)} = t_{I \setminus \{i_n\}}^{(n-2)}(t_I^{(n-1)}, t^{(n)},\ldots,t^{(M)}) \, ,
\eez
and corresponding expressions for $t_I^{(r)}$, $r=1,\ldots,n-3$. 
\item $\mathcal{P}_\Omega$ is a common point of all $\mathcal{P}_I$, $I \subset \Omega$. 
According to (\ref{DM_t_I^n-1__h_r}), on $\mathcal{P}_\Omega$ we have
\bez
      t_\Omega^{(M)} = - c_\Omega \; .
\eez
Clearly, $\mathcal{P}_\Omega = \mathcal{U}_\Omega$, and is thus visible. 
\item The hyperplane $\mathcal{P}_{\{i_1,i_2\}}$ is given by 
\bez
     t^{(1)}_{\{i_1,i_2\}} = - (p_{i_1} + p_{i_2}) \, t^{(2)} - \tilde{c}_{\{i_1,i_2\}} \, ,
\eez
and we have 
\be
    \theta_{i_2} - \theta_{i_1} = (p_{i_2} - p_{i_1}) \, (t^{(1)} - t^{(1)}_{\{i_1,i_2\}}) \, ,
        \label{DM_theta_diff_12}
\ee
so that, for $i_1 < i_2$,
\bez
    t^{(1)} \lessgtr t^{(1)}_{\{i_1,i_2\}} 
    \quad \Longleftrightarrow \quad \theta_{i_2} \lessgtr \theta_{i_1} \; .
\eez
Together with (\ref{DM_cU_i}), this implies in particular that each $\mathcal{U}_i$ is the intersection 
of half-spaces, and thus a closed convex set. None of these sets is empty since they 
all contain $\mathcal{U}_\Omega$. It follows in turn that each set $\mathcal{U}_I$ (with non-empty $I$) is 
non-empty, closed and convex (and thus in particular connected). 
\end{itemize}

\begin{figure}[H] 
\begin{center}
\includegraphics[scale=.45]{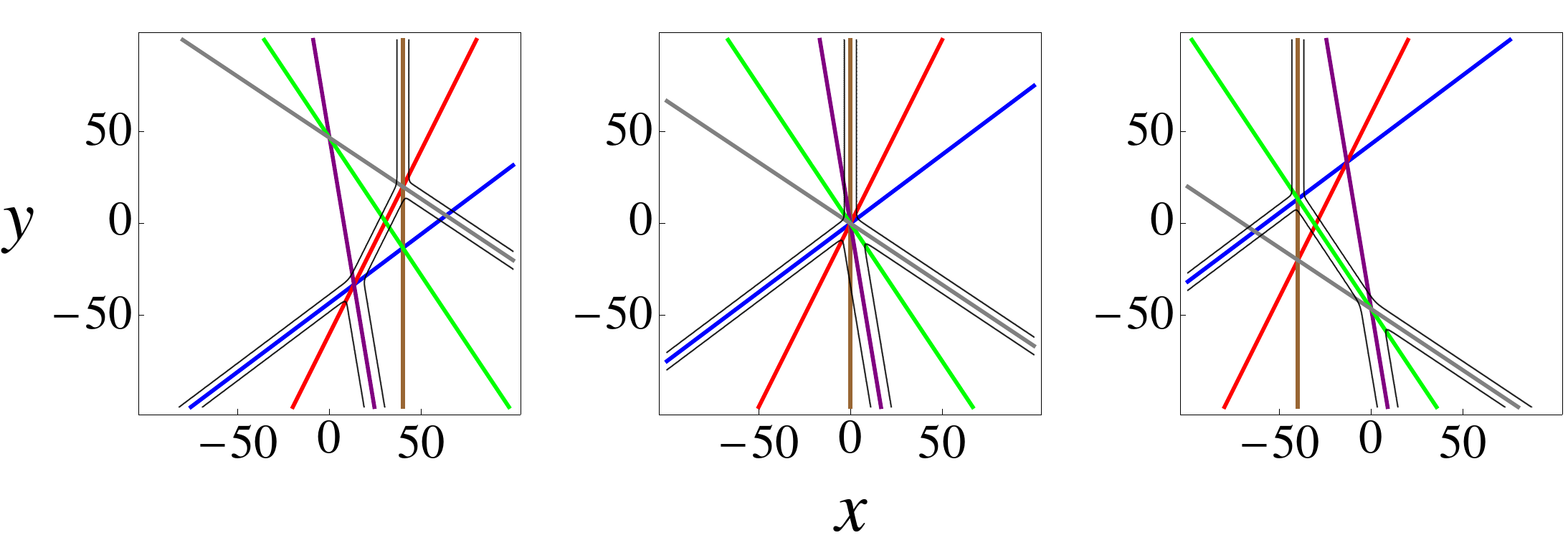} 
\parbox{15cm}{
\caption{A soliton solution with $M=3$, hence $\Omega = \{1,2,3,4\}$, at times 
$t<t_\Omega$, $t=t_\Omega$ and $t>t_\Omega$. A thin line is the coincidence of two phases. 
It corresponds to some $\mathcal{P}_{\{i,j\}}$ restricted to the respective value of $t$. 
Only a thick part of such a line is visible at the respective value of time. It 
corresponds to some $\mathcal{U}_{\{i,j\}}$, restricted to that value of $t$. 
The left and also the right plot shows two visible coincidences of three phases, corresponding 
to points in some $\mathcal{U}_{\{i,j,k\}}$. They coincide in the middle plot to form 
a visible four phase coincidence, the point $\mathcal{U}_\Omega$. \label{DM_fig:M3plot} }
}
\end{center} 
\end{figure}

\begin{lemma}[\cite{DM_DMH11KPT}, Proposition~A.3]
For $K=\{k_1,\ldots,k_{n+1}\}$, $n \in [M]$, we have\footnote{For $n=1$, this is 
(\ref{DM_theta_diff_12}), since $t^{(0)}_i = -\theta_i$. }
\bez
   t^{(n-1)}_{K \setminus \{k_j\}} - t^{(n-1)}_{K \setminus \{k_l\}}
 = (p_{k_j} - p_{k_l})(t^{(n)} - t^{(n)}_K) 
   \qquad \quad j,l \in \{1,\ldots,n+1\}   \; .
\eez
\end{lemma}

With the help of this important lemma, we obtain the following result. 

\begin{proposition}
\label{DM_prop:Bruhat}
If $K=\{k_1,\ldots,k_{n+1}\}$ is in linear order, i.e. $k_1 < k_2 < \cdots < k_{n+1}$, then 
\be
  \begin{array}{l}
   t^{(n-1)}_{K \setminus \{k_{n+1}\}} < t^{(n-1)}_{K \setminus \{k_n\}} < \cdots 
   < t^{(n-1)}_{K \setminus \{k_1\}}  \\[2ex]
   t^{(n-1)}_{K \setminus \{k_1\}} < t^{(n-1)}_{K \setminus \{k_2\}} < \cdots 
   < t^{(n-1)}_{K \setminus \{k_{n+1}\}} 
   \end{array} 
   \quad \mbox{for} \quad
   \begin{array}{l} 
     t^{(n)} < t^{(n)}_K \\[2ex]
     t^{(n)} > t^{(n)}_K \; .
   \end{array}     \label{DM_t^n-1_order}
\ee
\end{proposition}

The first chain in (\ref{DM_t^n-1_order}) is in 
\emph{lexicographic order}\index{order!lexicographic}, the second in 
\emph{reverse lexicographic order}, with respect to the index sets. 
This makes contact with \emph{higher Bruhat orders}\index{order!higher Bruhat}, see Section~\ref{DM_sec:hBT}. 
The following result is crucial for determining (non-)visible events. 

\begin{proposition}[\cite{DM_DMH11KPT}, Corollary~A.6] 
Let $I=\{i_1,\ldots, i_n\}$ and $k \in \Omega \setminus I$. On $\mathcal{P}_I$ we have
\be
    \theta_k - \theta_{i_1} = (p_k - p_{i_1}) \cdots (p_k - p_{i_n})
    \, (t^{(n)} - t^{(n)}_{I \cup \{k\}}) \; .   \label{DM_theta-diff-on-Pred}
\ee
\end{proposition}

\begin{example}
\label{DM_ex:n=2}
Let $n=2$ and thus $I=\{i,i'\}$ with $i < i'$. 
On $\mathcal{P}_I$, (\ref{DM_theta-diff-on-Pred}) reads 
\bez
   \theta_{k} -  \theta_{i} = (p_k - p_i) (p_k - p_{i'})
    \, (t^{(2)} - t^{(2)}_{\{i,i',k\}}) \; .
\eez
If $I$ is an interval, i.e. $i'=i+1$, we have either $k < i$ or $k > i+1$, 
and thus $\theta_k \lessgtr \theta_i$ iff $t^{(2)} \lessgtr t^{(2)}_{\{i,i+1,k\}}$. 
As a consequence, the part of $\mathcal{P}_{\{i,i+1\}}$ with $t^{(2)} \leq \min_k\{t^{(2)}_{\{i,i+1,k\}}\}$ 
is visible and the part with $t^{(2)} > \min_k\{ t^{(2)}_{\{i,i+1,k\}}\}$ is non-visible.
If $I$ is not an interval, then there is a $k \in \{2,\ldots,M\}$ such that 
$i < k < i'$. 
It follows that $\theta_k > \theta_i$ if $t^{(2)} < t^{(2)}_{\{i,i',k\}}$, so 
the part of $\mathcal{P}_{\{i,i'\}}$ with $t^{(2)} < \max_k\{t^{(2)}_{\{i,i',k\}} \, | \, i<k<i' \}$ is non-visible. 
If there is a $k \in \{1,\ldots,M+1\}$ with $k <i$ or $k >i'$, 
then the situation is as in the case of an interval. 
In the remaining case $I = \{1\} \cup \{M+1\}$, the part of $\mathcal{P}_{\{1,M+1\}}$ with 
$t^{(2)} \geq \max_k\{ t^{(2)}_{\{i,i',k\}}\}$ is visible. 
We conclude that $\mathcal{P}_I$, with $I$ of the form $\{i,i+1\}$, $i \in \{1,\ldots,M\}$, 
has a visible part extending to arbitrary negative values of $y=t^{(2)}$, and only 
$\mathcal{P}_{\{1,M+1\}}$ has a visible part extending to arbitrary positive values of $y$. 
Any visible part of another $\mathcal{P}_{\{i,i'\}}$ has to be bounded in the $xy$-plane. 
Furthermore, (\ref{DM_t^n-1_order}) shows that 
\bez
    t^{(1)}_{\{1,2\}} < t^{(1)}_{\{2,3\}} < \cdots < t^{(1)}_{\{M,M+1\}} \qquad
    \mbox{for} \qquad t^{(2)} < \min_{i,k}\{t^{(2)}_{\{i,i+1,k\}}\} \; .
\eez 
All this information determines the asymptotic line soliton structure in the $xy$-plane 
depicted in Figure~\ref{DM_fig:asympt}. 
\end{example}

\begin{SCfigure}[2.][hbtp]
\includegraphics[scale=.36]{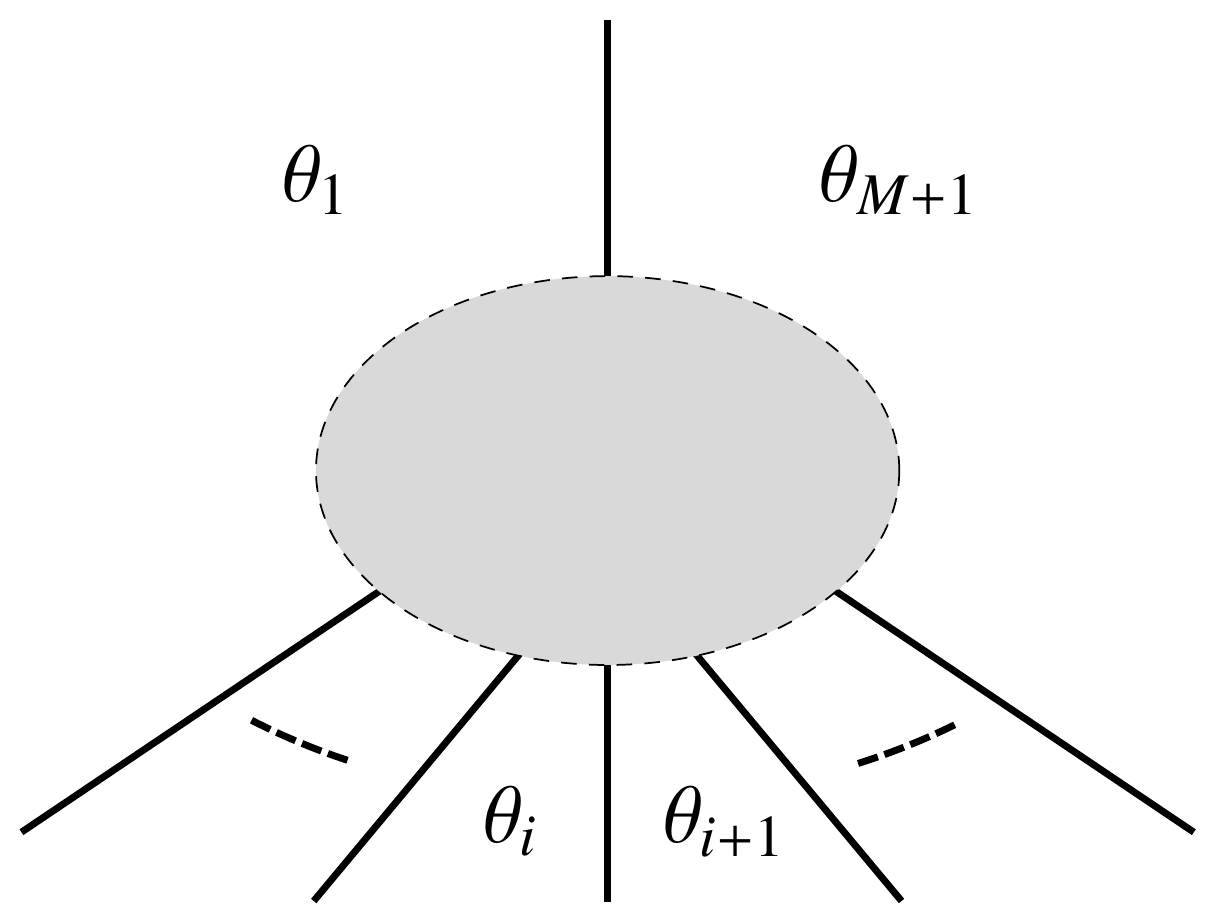} 
\caption{Asymptotic structure in the $xy$-plane ($x=t^{(1)}$ horizontal, $y=t^{(2)}$ vertical 
coordinate) for the line soliton solutions given by (\ref{DM_tau_simple_class}). Outside a 
large enough disk, the $xy$-plane is divided into regions as shown in the figure, where one 
of the phases $\theta_i$ dominates all others. 
This structure is independent of the values of $t^{(r)}$, $r>2$.  }
\label{DM_fig:asympt} 
\end{SCfigure} 

For $k \leq l$, let $[k,l]$ denote the \emph{interval} $\{k,k+1,\ldots,l\}$. We call an interval 
\emph{even} (respectively \emph{odd}) if its cardinality is even (respectively odd). 
Now we formulate a generalization of results in Example~\ref{DM_ex:n=2}. The proof is by 
inspection of (\ref{DM_theta-diff-on-Pred}), which depends on the structure of $I$. 
We note that an even interval 
cannot influence the sign of the right hand side of (\ref{DM_theta-diff-on-Pred}). 

\begin{proposition}
\label{DM_prop:asymp_vis}
Let $I$ be an $n$-subset of $\Omega = [M+1]$, $1< n < M+1$. \\
(1) Let $I$ be the disjoint union of even intervals, and also $\{1\}$ if $n$ is odd.
Then 
\bez
 \big\{ \mathbf{t} \in \mathcal{P}_I \, | \, t^{(n)} \leq \min\{ t^{(n)}_K | \, 
        K \in {\Omega \choose n+1}, \, I \subset K \}\, \big\}
\eez
is visible, but its complement 
\bez
  \big\{\mathbf{t} \in \mathcal{P}_I \, | \, t^{(n)} > 
  \min\{ t^{(n)}_K | \, K \in {\Omega \choose n+1}, \, I \subset K \} \, \big\}
\eez 
not.   \\
(2) Let $I$ be the disjoint union of $\{M+1\}$ and any number of even intervals, and also $\{1\}$ if $n$ is even. 
Then 
\bez 
  \big\{ \mathbf{t} \in \mathcal{P}_I \, | \, t^{(n)} \geq \max\{ t^{(n)}_K | \, K \in {\Omega \choose n+1}, \, I \subset K \} \, \big\}
\eez
is visible, but its complement 
\bez
 \big\{ \mathbf{t} \in \mathcal{P}_I \, | \, t^{(n)} < \max\{ t^{(n)}_K | \, K \in {\Omega \choose n+1}, 
     \, I \subset K \} \, \big\}
\eez
not. \\
(3) If $I$ is not of the form specified in (1) or (2), a visible part of $\mathcal{P}_I$ 
can only appear for $t^{(n)}$ between $\min\{ t^{(n)}_K | \, K \in {\Omega \choose n+1}, \, I \subset K \}$ 
and $\max\{ t^{(n)}_K | \, K \in {\Omega \choose n+1}, \, I \subset K \}$.
\end{proposition}

\begin{example}
Let $n=3$. According to Proposition~\ref{DM_prop:asymp_vis}, 
for $t^{(3)} \leq \min\{ t^{(3)}_K | \, K \in {\Omega \choose 4} \}$ only (the corresponding 
parts of) $\mathcal{P}_{\{1,i,i+1\}}$, $i = 2,\ldots,M$, are visible,
and for $t^{(3)} \geq \max\{ t^{(3)}_K | \, K \in {\Omega \choose 4} \}$ only (the corresponding 
parts of) $\mathcal{P}_{\{i,i+1,M+1\}}$, $i=1,\ldots,M-1$, are visible. 
The order of the values $t^{(2)}_{\{1,i,i+1\}}$, respectively $t^{(2)}_{\{i,i+1,M+1\}}$,  
follows from (\ref{DM_t^n-1_order}). All this leads to the structure of a 
line soliton solution for large negative time (left comb), respectively large positive time (right comb), 
shown in Figure~\ref{DM_fig:asympt_tree}. 
\begin{figure}
\begin{center}
\includegraphics[scale=.6]{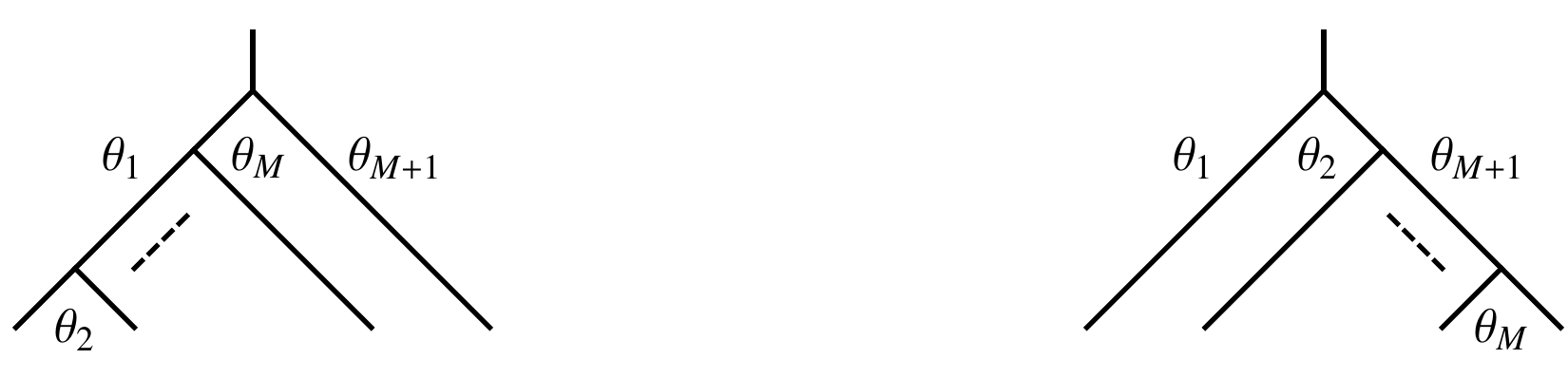}  
\parbox{15cm}{
\caption{For sufficiently large negative (respectively positive) values of time $t=t^{(3)}$, 
any line soliton solution from the class (\ref{DM_tau_simple_class}) has the tree shape 
in the $xy$-plane shown by the left (right) graph. \label{DM_fig:asympt_tree} } 
}
\end{center} 
\end{figure} 
\end{example}

\begin{example}
Let $M=5$ and $n=4$. 
For $t^{(4)} \leq \min\{ t^{(4)}_K | \, K \in {[6] \choose 5} \}$, the corresponding 
part of $\mathcal{P}_I$ is only visible if $I$ is one of the sets $\{1,2,3,4\}$, 
$\{1,2,4,5\}$, $\{1,2,5,6\}$, $\{2,3,4,5\}$, $\{2,3,5,6\}$, $\{3,4,5,6\}$. 
For $t^{(4)} \geq \max\{ t^{(4)}_K | \, K \in {[6] \choose 5} \}$, the corresponding 
part of $\mathcal{P}_I$ is only visible if $I$ is one of $\{1,2,3,6\}$, 
$\{1,3,4,6\}$, $\{1,4,5,6\}$. Also see Figure~\ref{DM_fig:T64chains}. 
\end{example}

In the following, $K$ denotes the set $\{ k_1,\ldots, k_{n+1} \}$, and will be assumed 
to be in linear order, so that $k_1 < \cdots < k_{n+1}$. We split such a set 
into\footnote{$\lceil n/2 \rceil$ denotes the smallest integer greater than or 
equal to $n/2$, and $\lfloor n/2 \rfloor$ the largest integer smaller than or equal to $n/2$. }
\bez
    && K_< =\{ k_{n-2r} \, | \, r=0,\ldots,\lceil n/2 \rceil-1 \} 
           =\{ k_{2-(n \, \mathrm{mod} \, 2)}, \ldots, k_n \}    \, , \\ 
    && K_> =\{ k_{n-2r+1} \, | \, r=0,\ldots,\lfloor n/2 \rfloor \} 
           =\{ k_{1+(n \, \mathrm{mod} \, 2)}  ,\ldots,k_{n+1} \}        \; . 
\eez
Choosing a point $\mathbf{t}_0 \in \mathcal{P}_K$ means fixing the free coordinates on $\mathcal{P}_K$ to 
values $t^{(n+1)}_0, \ldots,t^{(M)}_0$. For each $k \in K$, it determines a line 
\bez
 \{ \mathbf{t}_{K \setminus \{k\}}(\lambda,t^{(n+1)}_0, \ldots, t^{(M)}_0) \, | 
 \, \lambda \in \mathbb{R} \} \subset \mathcal{P}_{K \setminus \{k\}} \; . 
\eez

\begin{proposition}[\cite{DM_DMH11KPT}, Proposition~A.7]
\label{DM_prop:invisible_half-line}
The following half-lines are non-visible:  
\bez
 && \{ \mathbf{t}_{K \setminus \{k\}}(\lambda,t^{(n+1)}_0, \ldots, t^{(M)}_0) \, | \, 
     \lambda < t^{(n)}_K(t^{(n+1)}_0, \ldots, t^{(M)}_0) \} \; \mbox{ for } \, k \in K_< \; , \\
 && \{ \mathbf{t}_{K \setminus \{k\}}(\lambda,t^{(n+1)}_0, \ldots, t^{(M)}_0) \, | \, 
     \lambda > t^{(n)}_K(t^{(n+1)}_0, \ldots, t^{(M)}_0) \} \; \mbox{ for } \, k \in K_> \; .
\eez
\end{proposition}

\begin{proposition}
\label{DM_prop:non-visible}
If $\mathbf{t}_0 = \mathbf{t}_I(t^{(n)}_0, \ldots, t^{(M)}_0)$ is non-visible, then there is 
a $k \in \Omega \setminus I$ such that $\mathbf{t}_0$ lies on the non-visible side 
of the point $\mathbf{t}_{I \cup \{k\}}(t^{(n+1)}_0, \ldots, t^{(M)}_0)$ on the line 
$\{\mathbf{t}_I(\lambda,t^{(n+1)}_0, \ldots, t^{(M)}_0) \, | \, \lambda \in \mathbb{R} \}$.
\end{proposition}
\begin{proof} 
Since $\mathbf{t}_0$ is non-visible, it lies in some $\mathcal{U}_k$, $k \notin I$. 
Let $K=I \cup \{k\}$. Then $\mathbf{t}_1 := \mathbf{t}_K(t^{(n+1)}_0, \ldots, t^{(M)}_0)$
lies on the above line. Let us consider the case $t^{(n)}_0 < t^{(n)}_K(t^{(n+1)}_0, \ldots, t^{(M)}_0)$. 
If $k \in K_>$, then 
(\ref{DM_theta-diff-on-Pred}) shows that $\theta_k < \theta_{i_1}$, which contradicts 
$\mathbf{t}_0 \in \mathcal{U}_k$. Hence $k \in K_<$ and $\mathbf{t}_0$ lies on the non-visible 
side of $\mathbf{t}_1$ according to Proposition~\ref{DM_prop:invisible_half-line}. 
A similar argument applies in the case $t^{(n)}_0 > t^{(n)}_K(t^{(n+1)}_0, \ldots, t^{(M)}_0)$.
\end{proof}

According to Proposition~\ref{DM_prop:non-visible}, Proposition~\ref{DM_prop:invisible_half-line} provides 
us with a method to determine \emph{all} non-visible points, and thus also all visible points. 
\vskip.2cm

A point $\mathbf{t}_0 \in \mathcal{P}_I$ is called \emph{generic} if $\mathbf{t}_0 \notin \mathcal{P}_J$ 
for every $J \subset \Omega$ , $J \not\subset I$.

\begin{proposition}
Let $\mathbf{t}_0 \in \mathcal{P}_I$ be generic and visible, and $U$ any convex neighborhood 
of $\mathbf{t}_0$ in $\mathcal{P}_I$ that does not intersect any $\mathcal{P}_K$ with $I \subset K$, 
$I \neq K$. 
Then $U$ is visible.
\end{proposition}
\begin{proof} Let $U$ be a neighborhood as specified in the assumptions. Suppose a non-visible point 
$\mathbf{t}_1 \in U$ exists. Then there is some $k \in \Omega \setminus I$ such that, at $\mathbf{t}_1$,  
$\theta_k$ dominates all phases associated with elements of $I$. 
Let $\mathbf{t}'$ be the point 
where the line segment connecting $\mathbf{t}_0$ and $\mathbf{t}_1$ intersects the boundary 
of $\mathcal{U}_k$. Then $\mathbf{t}' \in U \cap \mathcal{U}_k \subset \mathcal{P}_{I \cup \{k\}}$ contradicts 
one of our assumptions.
\end{proof}

\begin{proposition}
\label{DM_prop:visible_points}
If $\mathbf{t}_0 \in \mathcal{P}_K$ is generic and visible, 
then points sufficiently close to $\mathbf{t}_0$ 
on the complementary half-line of any of the half-lines in Proposition~\ref{DM_prop:invisible_half-line}
are also visible. 
\end{proposition}
\begin{proof}
We have $\mathbf{t}_0 \in \mathcal{U}_{k_r}$, $r=1,\ldots, n+1$. The lines 
\bez
   \mathcal{L}_r = \{ \mathbf{t}_{K \setminus \{k_r\}}(\lambda,t^{(n+1)}_0, \ldots, t^{(M)}_0) \, | \, 
     \lambda \in \mathbb{R} \}  \qquad \quad r=1,\ldots, n+1
\eez 
all lie in the $n$-dimensional space $\mathcal{E}$ defined by $t^{(s)} = t^{(s)}_0$, $s=n+1,\ldots,M$. 
Since $\mathbf{t}_0$ is assumed to be visible and generic, there is a neighborhood of $\mathbf{t}_0$ 
covered by the sets $\mathcal{U}_{k_r} \cap \mathcal{E}$. Since each line contains the visible point 
$\mathbf{t}_0$, its visible part extends on the complementary side of that in Proposition~\ref{DM_prop:invisible_half-line}, either indefinitely or until it 
meets some $\mathcal{U}_m$ with $m \notin K$. 
\end{proof}

\begin{proposition}
Let $I \in {\Omega \choose n}$, $n \in \{1,\ldots,M\}$, and $t^{(n+1)}_0, \ldots, t^{(M)}_0 \in \mathbb{R}$.  \\
If all points 
$\mathbf{t}_{I \cup \{k\}}(t^{(n+1)}_0, \ldots, t^{(M)}_0)$, $k \in \Omega \setminus I$,
are non-visible, then the whole line 
\bez
 \{ \mathbf{t}_I(\lambda,t^{(n+1)}_0, \ldots, t^{(M)}_0) \, | \, 
     \lambda \in \mathbb{R} \}
\eez 
is non-visible. 
\end{proposition}
\begin{proof}
Suppose there is a visible point 
$\mathbf{t}_0 = \mathbf{t}_I(\lambda_0,t^{(n+1)}_0, \ldots, t^{(M)}_0)$ on the 
above line, which we denote as $\mathcal{L}$. Let 
$\mathbf{t}_1 := \mathbf{t}_{I \cup \{m\}}(t^{(n+1)}_0, \ldots, t^{(M)}_0)$
be the nearest of the non-visible points specified in the assumption. 
Then there is some $m' \in \Omega \setminus I$, $m' \neq m$, such that 
$ \mathbf{t}_1 \in \mathcal{U}_{m'} \cap \mathcal{L}$. 
The line segment between $\mathbf{t}_0$ and $\mathbf{t}_1$ meets the boundary 
of the convex set $\mathcal{U}_{m'}$ at the point 
$\mathbf{t}_{I \cup \{m'\}}(t^{(n+1)}_0, \ldots, t^{(M)}_0)$. Since the latter 
would then be visible, we have a contradiction.
\end{proof}

\begin{proposition}
For each $I \in {\Omega \choose n}$, $n \in \{1,\ldots,M\}$, there are 
$t^{(n+1)}_0, \ldots, t^{(M)}_0 \in \mathbb{R}$ such that the line 
$\{ \mathbf{t}_I(\lambda,t^{(n+1)}_0, \ldots, t^{(M)}_0) \, | \, 
     \lambda \in \mathbb{R} \}$ 
has a visible part. 
\end{proposition}
\begin{proof} 
Since $\mathcal{P}_\Omega$ is a visible point, we can use 
Proposition~\ref{DM_prop:visible_points} iteratively.
\end{proof}

Now we arrived at the following situation. 
 From Proposition~\ref{DM_prop:Bruhat}, we recall that 
\be
  \begin{array}{l}
   t^{(n-1)}_{K \setminus \{k_{n+1}\}} < t^{(n-1)}_{K \setminus \{k_{n-1}\}} < \cdots 
   < t^{(n-1)}_{K \setminus \{k_{1 + (n \, \mathrm{mod} \, 2)}\}}  \\[2ex]
   t^{(n-1)}_{K \setminus \{k_{2-(n \, \mathrm{mod} \, 2)}\}} < \cdots 
   < t^{(n-1)}_{K \setminus \{k_{n-2}\}} < t^{(n-1)}_{K \setminus \{k_n\}}
   \end{array} 
   \quad \mbox{for} \quad
   \begin{array}{l} 
     t^{(n)} < t^{(n)}_K \\[2ex]
     t^{(n)} > t^{(n)}_K \; .
   \end{array}     \label{DM_t^n-1_order_visible}
\ee
For all $t^{(n-1)}_{K \setminus \{k_i\}}$ that are absent in the respective 
chain, there is no visible event in the respective half-space ($t^{(n)} < t^{(n)}_K$, 
respectively $t^{(n)} > t^{(n)}_K$). 
Let $\mathbf{t}_0 \in \mathcal{P}_K$ be given by $t^{(n)} = t^{(n)}_K$ and fixing the higher 
variables to $t^{(n+1)}_0, \ldots, t^{(M)}_0$. 
Let $\mathbf{t}_0$ be visible and generic, and $t^{(n-1)}_{K \setminus \{k_i\}}$ in 
one of the chains in (\ref{DM_t^n-1_order_visible}). 
For $t^{(n)}$ close enough to $t^{(n)}_K$, and on the respective side according 
to (\ref{DM_t^n-1_order_visible}), every event in $\mathcal{P}_{K \setminus \{k_i\}}$ 
with remaining coordinates $t^{(n)}, t^{(n+1)}_0, \ldots, t^{(M)}_0$ is visible.

\begin{example}
For $n=3$ and $k_1 < k_2 < k_3 < k_4$, (\ref{DM_t^n-1_order_visible}) takes the form
\bez
    \begin{array}{l}
    y_{\{k_1,k_2,k_3\}} < y_{\{k_1,k_3,k_4\}} \\
    y_{\{k_2,k_3,k_4\}} < y_{\{k_1,k_2,k_4\}}
    \end{array} \quad \mbox{if} \quad 
    t \, \lessgtr \, t_{\{k_1,k_2,k_3,k_4\}} \; .
\eez
Fixing all variables $t^{(n)}$, $n>3$, for $t<t_{\{k_1,k_2,k_3,k_4\}}$  
close enough to $t_{\{k_1,k_2,k_3,k_4\}}$, the corresponding 
points of $\mathcal{P}_{\{k_1,k_2,k_3\}}$ and $\mathcal{P}_{\{k_1,k_3,k_4\}}$ are visible, 
but not the corresponding points of $\mathcal{P}_{\{k_2,k_3,k_4\}}$ and $\mathcal{P}_{\{k_1,k_2,k_4\}}$,
whereas for $t>t_{\{k_1,k_2,k_3,k_4\}}$ it is the other way around. 
All this describes a \emph{tree rotation}, see the plots in Figure~\ref{DM_fig:rotation}. 
\end{example}

We described line solitons as objects moving in the $xy$-plane (where $x=t^{(1)}$ and 
$y=t^{(2)}$). They evolve according to the KP equation\index{KP!equation} (with evolution parameter 
$t=t^{(3)}$), the first equation of the KP hierarchy\index{KP!hierarchy}. A higher KP hierarchy equation 
has one of the parameters $t^{(r)}$, $r>3$, as its evolution parameter. We do not 
consider the corresponding evolutions in this work, but it turned out that these 
parameters are important in order to classify the various evolutions. 
We showed that solitons from a subclass have the form of rooted 
binary trees with leaves extending to infinity in the $xy$-plane and evolving 
by right rotation as time $t$ proceeds. They all start with the same asymptotic form 
as $t \sim -\infty$ and end with the same asymptotic form as $t \sim +\infty$. These are 
the maximal and minimal element, respectively, of a Tamari lattice\index{Tamari!lattice}, 
and any generic evolution 
thus corresponds to a maximal chain. For a soliton configuration with $M+1$ leaves, 
which chain is realized depends on the values of the parameters $t^{(r)}$, $r=1,\ldots,M$.

\section{Higher Bruhat and higher Tamari orders}
\label{DM_sec:hBT}
According to Proposition~\ref{DM_prop:Bruhat}, a substantial role in the combinatorics 
underlying the tree-shaped line solitons is played by the order relations (\ref{DM_t^n-1_order}). 
They are at the roots of the generalization by Manin and Schechtman of the 
weak Bruhat order\index{order!weak} 
on the set of permutations of $[m]$ to `higher Bruhat orders'\index{order!higher Bruhat} 
\cite{DM_Manin+Schecht86a,DM_Manin+Shekhtman86b,DM_Manin+Schecht89}, also see \cite{DM_Ziegler93}. 
In the following subsection we recall some definitions and results mainly from \cite{DM_Ziegler93}. 
In section~\ref{DM_subsec:hT} we introduce `higher Tamari orders'.

\subsection{Higher Bruhat orders}
\label{DM_subsec:hB}
Let $n,N \in \mathbb{N}$ with $1 \leq n \leq N-1$. An element $K \in {[N] \choose n+1}$ will 
be written as $K=\{k_1, \ldots, k_{n+1}\}$ with $k_1 < \cdots < k_{n+1}$. 
$P(K)$ denotes the \emph{packet}\index{packet} of $K$, i.e. the set of $n$-subsets of $K$. 
The \emph{lexicographic order}\index{order!lexicographic} on $P(K)$ is given by  
$K \setminus \{k_{n+1}\}, K \setminus \{k_n\}, \ldots, K \setminus \{k_1\}$. 
A \emph{beginning segment} of $P(K)$ has the form $\{ K \setminus \{k_{n+1}\}, K \setminus \{k_n\}, 
\ldots, K \setminus \{k_j\} \}$ for some $j$. An \emph{ending segment} is of the form 
$\{ K \setminus \{k_j\}, K \setminus \{k_{j-1}\}, \ldots, K \setminus \{k_1\} \}$.\footnote{$\emptyset$ 
and $P(K)$ are considered as being both, beginning and ending. }

A subset $U \subset {[N] \choose n+1}$ is called \emph{consistent} if its 
intersection with any $(n+1)$-packet\footnote{An $(n+1)$-packet is the packet of 
some element of ${[N] \choose n+2}$. } 
is either a beginning or an ending segment.
The \emph{higher Bruhat order}\index{order!higher Bruhat} $B(N,n)$ is the set of consistent 
subsets of ${[N] \choose n+1}$, 
ordered by single-step inclusion\index{single-step inclusion}\index{order!by inclusion}.\footnote{By 
a theorem of Ziegler \cite{DM_Ziegler93}, this 
definition is equivalent to the original one of Manin and Schechtman 
 \cite{DM_Manin+Schecht86a,DM_Manin+Shekhtman86b},
also see \cite{DM_Felsner+Weil00}. For finite sets $U,U'$, \emph{single-step inclusion} 
is defined by $U \subset U'$ and $|U'|=|U|+1$. }

\begin{example}
The consistent subsets of ${[3] \choose 2}$ are $\emptyset$, $\{\{1,2\}\}$, $\{\{2,3\}\}$, 
$\{\{1,2\}, \{1,3\} \}$, $\{ \{1,3\}, \{2,3\} \}$, 
$\{\{1,2\}, \{1,3\}, \{2,3\}\}$. 
Single-step inclusion leads to $B(3,1)$, which has a hexagonal Hasse diagram\index{Hasse diagram} 
(also see Figure~\ref{DM_fig:B31,T31} in Section~\ref{DM_subsec:hT}). 
\end{example}

A linear order $\rho$ on ${[N] \choose n}$ (which may be regarded as a permutation of ${[N] \choose n}$) 
is called \emph{admissible} if, for every $K \in {[N] \choose n+1}$,  
the packet of $K$ appears in it either in lexicographic\index{order!lexicographic} 
or in reverse lexicographic order. 
Let $A(N,n)$ be the set of admissible linear orders of ${[N] \choose n}$. 
Two elements $\rho, \rho'$ of $A(N,n)$ are \emph{elementarily equivalent},  
if they differ only by the exchange of two neighboring elements that are not contained 
in a common packet. The resulting equivalence relation will be denoted by $\sim$. 
For each $\rho \in A(N,n)$, the \emph{inversion set}\index{inversion set} $\mathrm{inv}(\rho)$ is the set
of all $K \in {[N] \choose n+1}$ for which $P(K)$ appears in reverse lexicographic order in $\rho$.
We have $\rho \sim \ \rho'$ iff $\mathrm{inv}(\rho) = \mathrm{inv}(\rho')$, so that the inversion 
set only depends on the equivalence class of $\rho$. All this results in a poset isomorphism
$U = \mathrm{inv}(\rho) \mapsto [\rho]$ between $B(N,n)$ and $A(N,n)/\!\!\sim$.\footnote{The difficult 
part is to show that the set of inversion sets coincides with the set of consistent sets, 
see \cite{DM_Ziegler93}. } 

For $[\rho] \in A(N,n)/\!\!\sim$, let $Q[\rho]$ be the intersection of all \emph{linear} orders 
in $[\rho]$, i.e. the partial order on ${[N] \choose n}$  given by $I' < I$ iff 
$I' <_\sigma I$ for all $\sigma \in [\rho]$. The set of linear extensions of $Q[\rho]$ 
coincides with $[\rho]$. We set $Q(U) := Q[\rho]$ where $U=\mathrm{inv}(\rho)$.

Of great help for the construction of higher Bruhat orders\index{order!higher Bruhat} is the existence 
\cite{DM_Manin+Schecht86a,DM_Manin+Shekhtman86b,DM_Ziegler93} of a natural bijection 
between $A(N,n)$ and the set of maximal chains of $B(N,n-1)$, which we describe next. 
With $\rho = (I_1,\ldots,I_s) \in A(N,n)$ (where $I_i \in {[N] \choose n}$ and $s={N \choose n}$) 
we associate the chain of consistent sets 
$\emptyset \to \{I_1\} \to \{I_1,I_2\} \to \cdots \to \{I_1,\ldots,I_s\} = {[N] \choose n}$
in $B(N,n-1)$. Conversely, given a maximal chain 
$\emptyset \to U_1 \to U_2 \to \cdots \to {[N] \choose n}$ 
of consistent sets in $B(N,n-1)$, these are ordered by single step inclusion\index{single-step inclusion}, hence  
$U_{r+1} \setminus U_r = \{I_r\}$ with some $I_r \in {[N] \choose n}$. Thus we obtain 
an admissible linear order $(I_1,\ldots,I_s) \in A(N,n)$. 
If, for some $K \in {[N] \choose n+1}$, $P(K)$ appears in (reverse) 
lexicographic order\index{order!lexicographic} in $\rho$, 
then $U_r \cap P(K)$ is a beginning (ending) segment, for $r=0,1,\ldots,s$. 
Conversely, if $U$ is an element of a maximal chain of $B(N,n-1)$, and if 
$U \cap P(K)$ is a beginning (ending) segment, then this holds for all 
elements of this chain. Thus  $P(K)$ appears in the corresponding $\rho$ in 
(reverse) lexicographic order\index{order!lexicographic}.
Furthermore, if two maximal chains have a common edge $U \stackrel{I}{\rightarrow} U \cup \{I\}$, 
then, obviously, both contain all packets having $I$ as a member in the same way (i.e. 
lexicographically, respectively reverse lexicographically). 

With the help of these results, all higher Bruhat orders\index{order!higher Bruhat} can be constructed iteratively, starting 
from the highest, i.e. $B(N,N-1)$. The latter consists of only two elements, 
$\emptyset$ and $\{[N]\}$.

Suppose we have constructed $B(N,n)$. Its elements are consistent sets of the form 
$\{K_1, \ldots, K_r\}$, where $K_i \in {[N] \choose n+1}$ 
and $r \in \{0,1,\ldots,{N \choose n+1}\}$ ($\emptyset$ if $r=0$). 
Associated with each such consistent set $U$ is an equivalence class $[\rho] \in A(N,n)/\!\!\sim$ 
such that $U=\mathrm{inv}(\rho)$. 
For each $\rho \in [\rho]$ we construct the corresponding maximal chain of $B(N,n-1)$, as explained 
above. 
The collection of all such maximal chains constitutes $B(N,n-1)$. Its elements are 
consistent sets $\{I_1,\ldots,I_r\}$, $r\in \{0,1,\ldots,s\}$. 
Now we can continue to construct $B(N,n-2)$, and so forth.

Having arrived at $B(N,1)$, the weak Bruhat order\index{order!weak} on the permutation group $S_N$, 
we can even proceed once more.
We consider a permutation $\pi = (\pi_1,\pi_2,\ldots,\pi_N) \in S_N$ as an order 
$\{\pi_1\} < \{\pi_2\} < \cdots < \{\pi_N\}$, which in turn determines 
the chain
$\emptyset \to \{\{\pi_1\}\} \to \{\{\pi_1\},\{\pi_2\}\}
\to \cdots \to \{\{\pi_1\},\ldots,\{\pi_N\}\} = {[N] \choose 1}$. 
These are the maximal chains of 
$B(N,0)$, which is isomorphic to the Boolean lattice\index{lattice!Boolean} of subsets of $[N]$ and forms an 
$N$-cube.\footnote{More generally, the elements of $B(N,n)$ 
can be represented as sets of $n$-faces of the $N$-cube \cite{DM_Voev+Kapr90,DM_Thomas03}. }

In the following, we represent a higher Bruhat order\index{order!higher Bruhat} $B(N,n)$ 
by a diagram, suppressing the labels of the vertices and expressing a maximal chain 
$\emptyset \to \{K_1\} \to \{K_1,K_2\} \to \cdots \to \{K_1,\ldots,K_q\} 
  = {[N] \choose n+1}$, $q={N \choose n+1}$, graphically as
\bez
     \bullet \stackrel{K_1}{\longrightarrow} \bullet \stackrel{K_2}{\longrightarrow} \bullet
     \; \cdots \; \bullet \stackrel{K_q}{\longrightarrow} \bullet \; .
\eez
The edges are thus labelled by the sets $K_i$, which are sequentially added to the preceding 
vertex, starting with the empty set. The vertices are thus given by sets of the form $\{K_1,\ldots,K_r\}$, $r \in \{0,1,\ldots,q\}$. 
They are ordered by single step inclusion\index{single-step inclusion}, 
and we have $(K_1,\ldots,K_q) \in A(N,n+1)$.

\begin{remark}
The weak Bruhat order\index{order!weak} $B(N,1)$ is a lattice and can be visualized 
as a polytope in $N-1$ Euclidean dimensions, called \emph{permutohedron}. 
Not all higher Bruhat orders\index{order!higher Bruhat} are lattices \cite{DM_Ziegler93} and not all 
can be realized as polytopes \cite{DM_Felsner+Ziegler01}.  
\end{remark}

In the context of KP line solitons\index{KP!soliton}, the relevance of higher Bruhat orders\index{order!higher Bruhat} is evident from Proposition~\ref{DM_prop:Bruhat}, 
as already mentioned there.
For fixed $M \in \mathbb{N}$ and parameters $p_i,c_i$, $i=1,\ldots,M+1$, the order 
$p_1 < \cdots < p_{M+1}$ induces an order on the `critical values' 
$t^{(n-1)}_J$, $J \in {\Omega \choose n}$, according to the following rule. For $t^{(n)} < t^{(n)}_K$,  
the values $t^{(n-1)}_J$, $J \in P(K)$, are ordered lexicographically, and for $t^{(n)} > t^{(n)}_K$
they are ordered reverse lexicographically. Via the bijection $I \mapsto t^{(|I|-1)}_I$ of 
subsets of $\Omega = [M+1]$ and the set of critical values, this corresponds to admissible 
permutations on ${\Omega \choose n}$. Without further restriction of the parameters, the 
resulting partial order is $B(M+1,n)$.

\subsubsection{How to obtain the poset $Q(U)$ for a consistent set $U$: an example}
\label{DM_sssec:B(4,2)}
This subsection explains in an elementary way the construction of the poset $Q(U)$ 
for a consistent set $U$ in the case $N=4$. There is only one 3-packet, namely 
\bez  
     {[4] \choose 3} = \{ \{1,2,3\}, \{1,2,4\}, \{1,3,4\}, \{2,3,4\} \} \, ,
\eez 
and thus the following 8 consistent subsets of ${[4] \choose 3}$, 
\bez
   && \emptyset, \; \{\{1,2,3\}\}, \; \{ \{1,2,3\}, \{1,2,4\} \}, 
      \; \{\{1,2,3\}, \{1,2,4\}, \{1,3,4\}\}, \\
   && \{\{1,2,3\}, \{1,2,4\}, \{1,3,4\}, \{2,3,4\}\}, \\ 
   && \{\{2,3,4\}\}, \; \{\{1,3,4\}, \{2,3,4\}\}, \; \{\{1,2,4\}, \{1,3,4\}, \{2,3,4\}\} \; .
\eez 
Single-step inclusion\index{single-step inclusion} results in $B(4,2)$, which has an octagonal 
Hasse diagram\index{Hasse diagram} 
(see Figure~\ref{DM_fig:B42Qs}).
The packets of the elements of ${[4] \choose 3}$ are given by 
\bez
  && P(\{1,2,3\})=\{ \{1,2\}, \{1,3\}, \{2,3\} \} \, , \quad
     P(\{1,2,4\})=\{ \{1,2\}, \{1,4\}, \{2,4\} \} \, , \\ 
  && P(\{1,3,4\})=\{ \{1,3\}, \{1,4\}, \{3,4\} \} \, , \quad 
     P(\{2,3,4\})=\{ \{2,3\}, \{2,4\}, \{3,4\} \} \;.
\eez 
For each consistent subset $U$ of ${[4] \choose 3}$, we consider a 
table, which displays the packet of each element of $U$ downwards in reverse lexicographic 
order\index{order!lexicographic} (and the remaining packets in lexicographic order). 
The left one of the following two tables describes the case $U=\emptyset$, hence all packets are in 
lexicographic order. 
\begin{center}
\begin{tabular}[t]{c|c|c|c}
1 2 3 & 1 2 4 & 1 3 4 & 2 3 4 \\
\hline 
1 2 & 1 2 & 1 3 & 2 3 \\
1 3 & 1 4 & 1 4 & 2 4 \\
2 3 & 2 4 & 3 4 & 3 4 
\end{tabular}
\hspace{1cm}
\begin{tabular}[t]{c|c|c|c}
1 2 3 & 1 2 4 & 1 3 4 & 2 3 4 \\
\hline 
2 3 & 2 4 & 1 3 & 2 3 \\
1 3 & 1 4 & 1 4 & 2 4 \\
1 2 & 1 2 & 3 4 & 3 4 
\end{tabular}
\end{center}
 From the table, where e.g. $12$ stands for $\{1,2\}$, we read off\footnote{For example, the top 
node can only be an entry from the first row of the table. But the first column only leaves us 
with $12$. It is also obvious that $34$ is the bottom node. Furthermore, we see that from $13$ to $24$ 
there are two ways, via $14$ respectively $23$.}
cover relations and deduce a poset, drawn as a Hasse diagram\index{Hasse diagram}.
This leads to the very first poset (of both horizontal chains) in Figure~\ref{DM_fig:B42Qs}. 
It has two linear extensions, which are elements $\rho, \rho' \in A(4,2)$, one with $14, 23$, 
the other one with $23, 14$ instead. Since $14$ and $23$ belong to different packets, $\rho \sim \rho'$.
Since no packet is in reverse lexicographic order\index{order!lexicographic}, 
$[\rho] \in A(4,2)/\!\!\sim$ corresponds to the 
empty (consistent) set. The poset is $Q(\emptyset)$. 

\begin{figure}
\begin{center}
\includegraphics[scale=.3]{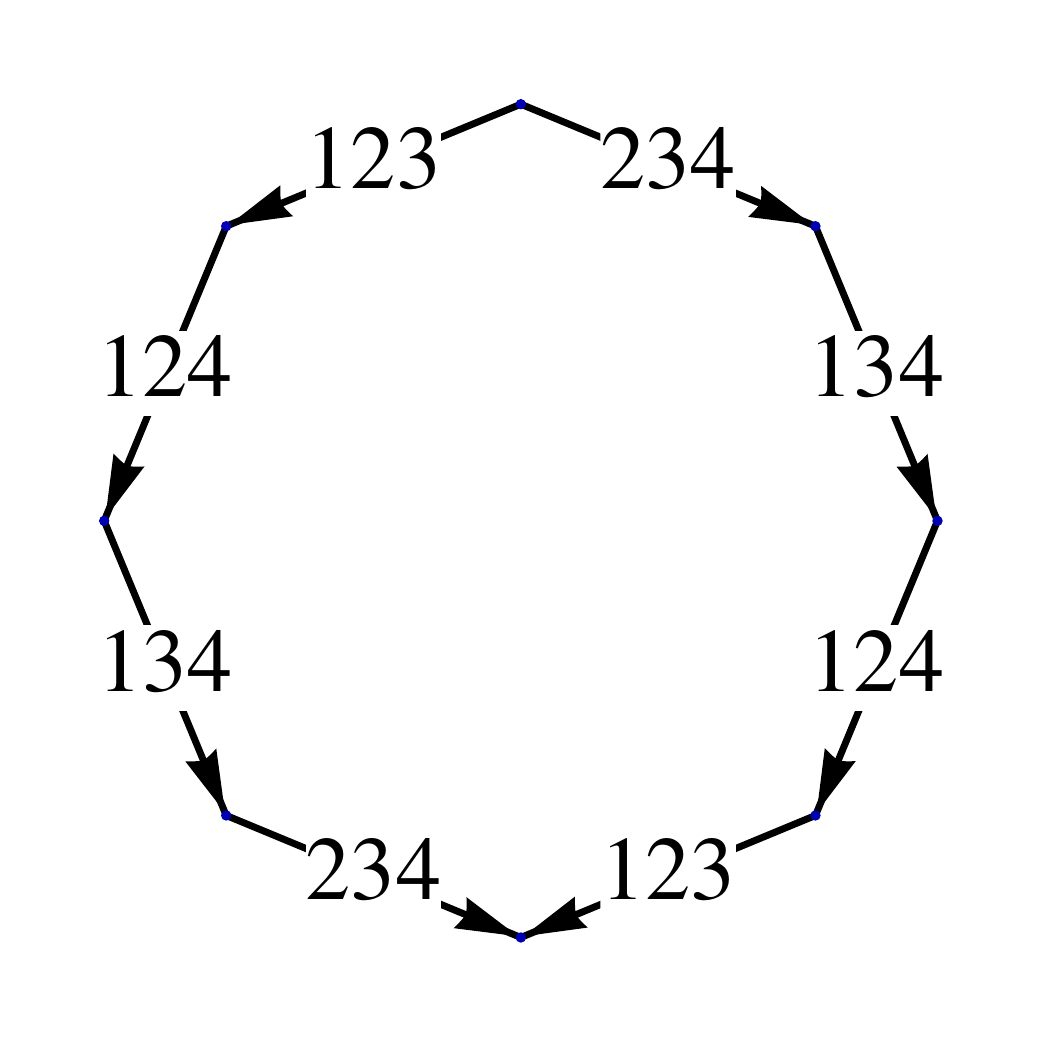}
\includegraphics[scale=.47]{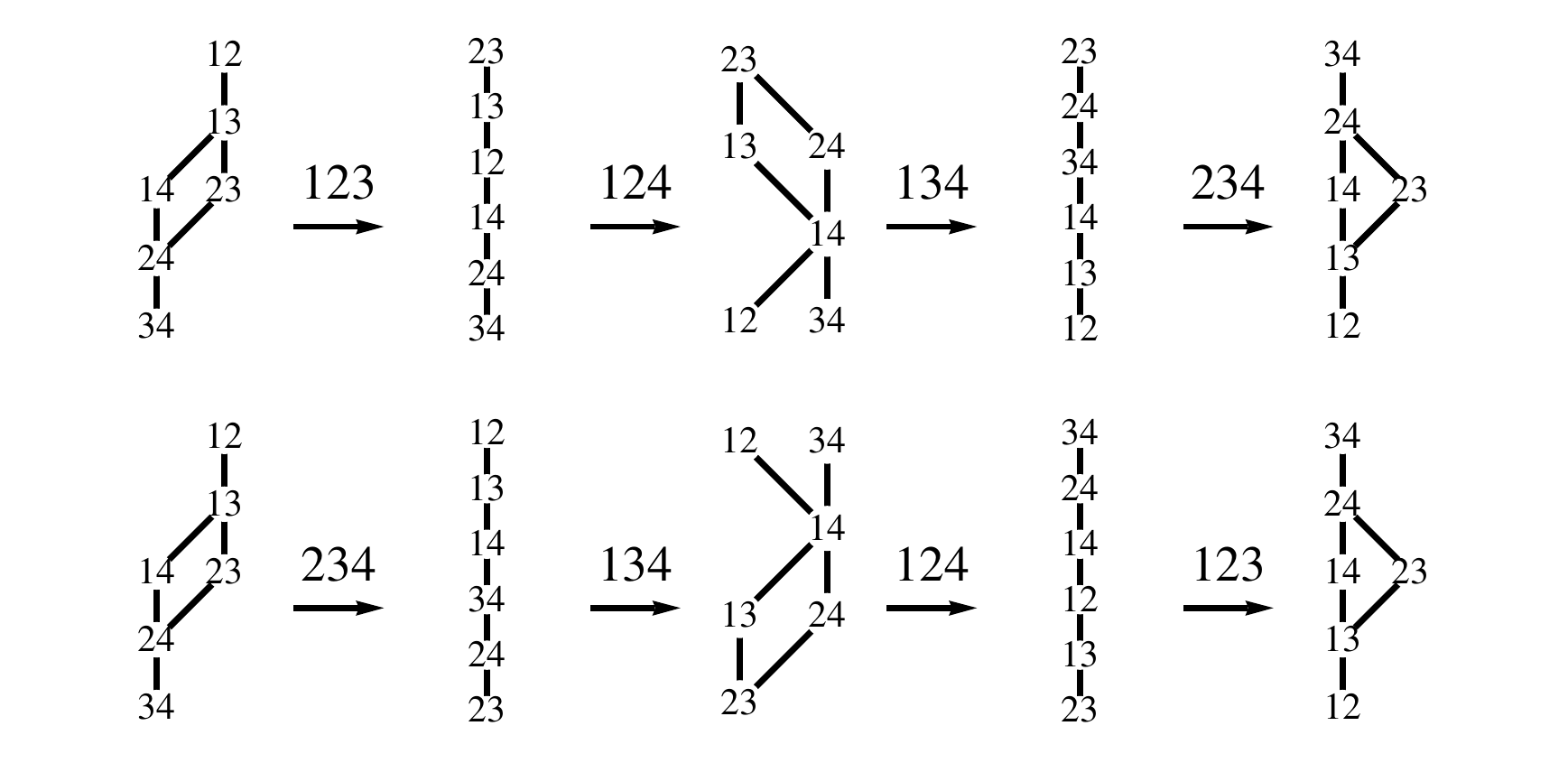} 
\parbox{15cm}{
\caption{$B(4,2)$ and its two maximal chains in terms of the $Q$-posets, which 
can be constructed in the way explained in section~\ref{DM_sssec:B(4,2)}. \label{DM_fig:B42Qs} }
}
\end{center} 
\end{figure}

The second table above decribes the case $U = \{ \{1,2,3\}, \{1,2,4\} \}$. Its evaluation 
leads to $Q(U)$, the third poset in the first horizontal chain in Figure~\ref{DM_fig:B42Qs}. It has
four linear extensions. Since $13$ and $24$, as well as $12$ and $34$, belong 
to different packets, all these extensions are equivalent, hence $U$ determines a single element 
of $A(4,2)/\!\!\sim$. 
Evaluating the remaining tables, we finally obtain the two chains of posets in Figure~\ref{DM_fig:B42Qs}.

\subsection{Higher Tamari orders}
\label{DM_subsec:hT}
Let $K \in {[N] \choose n+1}$. In the KP line soliton\index{KP!soliton} context in Section~\ref{DM_sec:tropical}, 
where $N=M+1$, we found the following rules concerning non-visible critical events. 
A point $\mathbf{t}_{K \setminus \{k\}}$ is non-visible if either (1) $t^{(n)} < t^{(n)}_K$ 
and $k \in K_<$, or (2) $t^{(n)} > t^{(n)}_K$ and $k \in K_>$. In the first case, the critical values 
$t^{(n)}_{K \setminus \{k\}}$ are ordered lexicographically, in the second case reverse lexicographically.
This induces a corresponding combinatorial rule on the sets that enumerate the critical values, 
and we can resolve the notion of `non-visibility' from the special KP line soliton\index{KP!soliton} 
context as follows. 

\begin{definition} 
\label{DM_def_non_vis}
Let $\rho \in A(N,n)$.
$I \in \rho$ is called \emph{non-visible} in $\rho$ if there is a $K \in {[N] \choose n+1}$ 
and a $k \in K$ such that $I=K \setminus \{k\}$, and if either 
$K \notin \mathrm{inv}(\rho)$ and $k \in K_<$, or 
$K \in \mathrm{inv}(\rho)$ and $k \in K_>$ (with $K_<$ and $K_>$ defined in 
Section~\ref{DM_sec:tropical}). 
$I$ is called \emph{visible} in $\rho$ if it is \emph{not} non-visible in $\rho$. 
\end{definition}

Since $\mathrm{inv}(\rho)$ only depends on the equivalence class of $\rho$, this definition 
induces a notion of non-visibility (visibility) in $[\rho] \in A(N,n)/\!\!\sim$.
Moreover, since an element of an admissible linear order $\rho \in A(N,n)$ corresponds to an edge 
$U \stackrel{I}{\rightarrow} U'$ of a maximal chain of $B(N,n-1)$ (cf. Section~\ref{DM_subsec:hB}), 
the above definition induces a notion of non-visibility of such an edge: 
$\; U \stackrel{I}{\rightarrow} U'$ is \emph{non-visible} in a maximal chain of $B(N,n-1)$ 
if $I = K \setminus \{k\}$ for some $K \in {[N] \choose n+1}$, and 
$k \in K_<$ if $U' \cap P(K)$ is a beginning segment, 
$k \in K_>$ if $U' \cap P(K)$ is an ending segment. 
If $U \stackrel{I}{\rightarrow} U'$ is non-visible in one maximal chain of $B(N,n-1)$, 
then it is non-visible in \emph{every} maximal chain that contains it. Therefore we can drop 
the reference to a maximal chain. 
\vskip.1cm

We have seen in Section~\ref{DM_subsec:hB} that any $U \in B(N,n)$ determines a poset $Q(U)$. 
Eliminating all non-visible elements (in admissible linear orders) of $Q(U)$, by application of the 
rules in the preceding definition, results in a subposet that we denote by $R(U)$. 

\begin{proposition}
\label{DM_prop:R(U)=R(U')}
An edge $U \stackrel{K}{\rightarrow} U'$ in $B(N,n)$ is non-visible iff $R(U) = R(U')$. 
\end{proposition}
\begin{proof} 
We will show that if $U \stackrel{K}{\rightarrow} U'$ is non-visible, then all elements of $P(K)$ 
are non-visible
in (any linear extension of) $Q(U)$ and $Q(U')$, and if $U \stackrel{K}{\rightarrow} U'$ is visible, 
then $R(U)$ and $R(U')$ differ by elements of $P(K)$. 
Let $I \in P(K)$, i.e. $K = I \cup \{k\}$ with some $k \in [N] \setminus I$,  
and $l \in [N] \setminus K$. We set $L = K \cup \{l\} \in {[N] \choose n+2}$ and 
$K' = L \setminus \{k\} \in {[N] \choose n+1}$. \\
Let $l \in L_<$. We recall that in this case \\
\hspace*{.3cm}(a) $U' \cap P(L)$ is a beginning segment if $U \stackrel{K}{\rightarrow} U'$ is non-visible, \\
\hspace*{.3cm}(b) $U' \cap P(L)$ is an ending segment if $U \stackrel{K}{\rightarrow} U'$ is visible. \\
If $k > l$, then $l \in {K'}_>$. In case (a) we have $K' \in U, U'$, hence 
$I = K' \setminus \{l\}$ is non-visible. In case (b) we have $K' \notin U, U'$, hence 
$I = K' \setminus \{l\}$ is \emph{not} non-visible with respect to $K'$. 
If $k < l$, then $l \in {K'}_<$. In case (a) we have $K' \notin U, U'$ and 
$I$ is again non-visible. In case (b) we have $K' \in U, U'$ and 
$I$ is again \emph{not} non-visible with respect to $K'$. 

 For $l \in L_>$, we have to exchange `beginning' and `ending' in the above conditions (a) and (b). 
If $k>l$, we have $l \in {K'}_<$. Then $I$ is 
non-visible in case (a) and not non-visible with respect to $K'$ in case (b). If $k<l$, then $l \in {K'}_>$. 
Again, $I$ is non-visible in case (a) and not non-visible with respect to $K'$ in case (b). 

We conclude from case (a) that non-visibility of $U \stackrel{K}{\rightarrow} U'$ implies that 
all elements of $P(K)$ are non-visible in $Q(U)$, as well as in $Q(U')$. 
Since $K$ and all the $K'$ exhaust the elements of ${[N] \choose n+1}$ whose packets contain $I$, 
it follows from case (b) that $I \in R(U)$ and $I \notin R(U')$ if $k \in K_>$, whereas 
$I \notin R(U)$ and $I \in R(U')$ if $k \in K_<$.
\end{proof}

\begin{example}
Let $N=6$, $n=3$ and $K=1346$, which stands for $\{1,3,4,6\}$. There are only two 
elements of ${[6] \choose 5}$ such that their packet contains $K$. These are 
$L_1 = 12346$ and $L_2=13456$, hence $l_1 = 2 \in (L_1)_<$ and $l_2 = 5 \in (L_2)_<$. \\
If $U \stackrel{K}{\rightarrow} U'$ is non-visible, then $U' \cap P(L_1)$ 
is a beginning segment: $U'\cap P(L_1)  = \{1234, 1236, 1246, 1346\}$.
The left table below displays the corresponding information for the construction 
of $Q(U')$ and $R(U')$, obtained via Definition~\ref{DM_def_non_vis}. Non-visible elements 
are marked in red. We see that the whole packet of $K$ (marked in the table in light red) is already 
non-visible as a consequence of the information obtained from the other elements of ${[6] \choose 4}$. 
Hence $R(U')=R(U)$.  \\
If $U \stackrel{K}{\rightarrow} U'$ is visible, then $U' \cap P(L_1)$ 
is an ending segment: $U'\cap P(L_1)  = \{1346, 2346\}$. In this case we obtain the right table. 
$R(U)$ contains the elements 134 and 146 of $P(K)$.
They are not present in $R(U')$, in which the new complementary elements 136 and 346 of $P(K)$ 
show up. For $L_2$ an analogous discussion applies. Finally we can conclude that $R(U') \neq R(U)$.
\definecolor{DM_LightRed}{rgb}{1.0,0.5,0.5}
\begin{center}
\begin{tabular}[t]{c|c|c|c|c}
1234 & 1236 & 1246 & 1346 & 2346 \\
\hline 
234 & 236 & 246 & {\color{DM_LightRed}346} & 234  \\
{\color{red}134} & {\color{red}136} & {\color{red}146} & {\color{DM_LightRed}146} & {\color{red}236} \\
124 & 126 & 126 & {\color{DM_LightRed}136} & 246 \\
{\color{red}123} & {\color{red}123} & {\color{red}124} & {\color{DM_LightRed}134} & {\color{red}346}
\end{tabular}
\hspace{1cm}
\begin{tabular}[t]{c|c|c|c|c}
1234 & 1236 & 1246 & 1346 & 2346 \\
\hline 
123 & 123 & 124 & {\color{cyan}346} & 346 \\
{\color{red}124} & {\color{red}126} & {\color{red}126} & {\color{red}146} & {\color{red}246} \\
134 & 136 & 146 & {\color{cyan}136} & 236 \\
{\color{red}234} & {\color{red}236} & {\color{red}246} & {\color{red}134} & {\color{red}234}
\end{tabular}
\end{center}
\end{example}

\begin{corollary}
\label{DM_cor:no_cycles}
If $U_0 \stackrel{K_1}{\rightarrow} U_1 \stackrel{K_2}{\rightarrow} \cdots 
\stackrel{K_r}{\rightarrow} U_r$ is any (not necessarily maximal) chain in $B(N,n)$, 
containing at least one visible edge, then $R(U_0) \neq R(U_r)$.
\end{corollary}
\begin{proof} 
Without restriction of generality we can assume that $U_0 \stackrel{K_1}{\rightarrow} U_1$ 
is visible and $K_1 = I \cup \{k\}$ with $k \in (K_1)_>$, so that $I \in R(U_0)$ 
and $I \notin R(U_1)$, see the proof of Proposition~\ref{DM_prop:R(U)=R(U')}.
Since $K_1 \in U_s$ for $s=1,\ldots,r$, and $k \in (K_1)_>$, $I$ does not appear in 
any $R(U_s)$, and in particular not in $R(U_r)$. 
\end{proof}

\begin{definition} 
The \emph{higher Tamari order}\index{order!higher Tamari} $T(N,n)$ is the poset with set of vertices $\{R(U) \, | \, U \in B(N,n) \}$ 
and the order given by $R(U) \leq R(U')$ if $U \leq U'$ in $B(N,n)$.
\end{definition}

\begin{remark}
The map $B(N,n) \to T(N,n)$, given by $U \mapsto R(U)$, is surjective and order preserving. 
The Tamari orders\index{order!higher Tamari} also inherit the following property from the 
Bruhat orders\index{order!higher Bruhat}. 
There is a bijection between the maximal chains of $T(N,n)$ and the linear extensions 
of the $R$-posets that form the vertices of $T(N,n+1)$. 
\end{remark}

In order to construct the Tamari order\index{order!higher Tamari} $T(N,n)$, 
in each of the maximal chains 
$\bullet \stackrel{K_1}{\longrightarrow} \bullet \stackrel{K_2}{\longrightarrow} \bullet
 \; \cdots \; \bullet \stackrel{K_q}{\longrightarrow} \bullet$ of $B(N,n)$ we locate 
the edges associated with non-visible $K$'s and eliminate them. This results in a reduced 
chain 
$\bullet \stackrel{K_{i_1}}{\longrightarrow} \bullet \stackrel{K_{i_2}}{\longrightarrow} \bullet
     \; \cdots \; \bullet \stackrel{K_{i_w}}{\longrightarrow} \bullet$.\footnote{The information 
that resides in the edge labels of the Tamari (i.e. reduced Bruhat) chains is not sufficient 
to construct the $R$-posets, which are the vertices of $T(N,n)$. The latter have to be constructed 
from the $Q$-posets via elimination of non-visible elements.  
}
Such an elimination involves identifying the two vertices that are connected by this edge 
in the Bruhat order\index{order!higher Bruhat}. The consistency of this identification is 
guaranteed by Proposition~\ref{DM_prop:R(U)=R(U')}. As a consequence of 
Corollary~\ref{DM_cor:no_cycles}, the elimination process cannot lead to cycles, so indeed 
defines a partial order. Figure~\ref{DM_fig:B31,T31} shows a simple example.
\begin{SCfigure}[1.6][hbtp]
\hspace{.5cm}
\includegraphics[scale=.7]{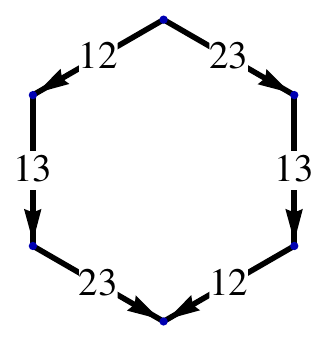} 
\hspace{1.cm}
\includegraphics[scale=.56]{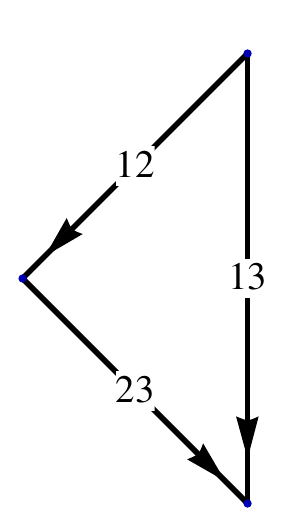}
\caption{$B(3,1)$ (weak Bruhat order\index{order!weak} on $S_3$), and the corresponding Tamari 
order\index{order!higher Tamari} $T(3,1)$. The latter is obtained from the former by eliminating in the left maximal 
chain the non-visible $13$ and in the right maximal chain the non-visible $12$ and $23$. }
\label{DM_fig:B31,T31} 
\end{SCfigure} 

Since $T(N,N-1) \cong B(N,N-1)$, both are represented by 
$\bullet \stackrel{[N]}{\longrightarrow} \bullet$.
 For $N=3$ and $N=4$, the vertices are represented in Figure~\ref{DM_fig:Miles_xy_B32_T32},
respectively Figure~\ref{DM_fig:B431,T43}, in terms of 
the $Q$-posets, respectively $R$-posets, and the information contained in the latter is translated 
into a soliton graph.
\begin{figure}[H] 
\begin{center}
\begin{minipage}[t]{3.5cm}
\vspace{-2.8cm}
\includegraphics[scale=.7]{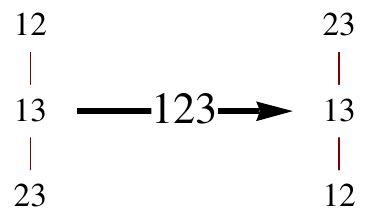}
\vskip.2cm
\includegraphics[scale=.7]{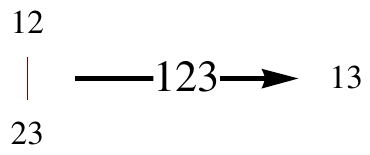} 
\end{minipage}
\hspace{1cm}
\includegraphics[scale=.34]{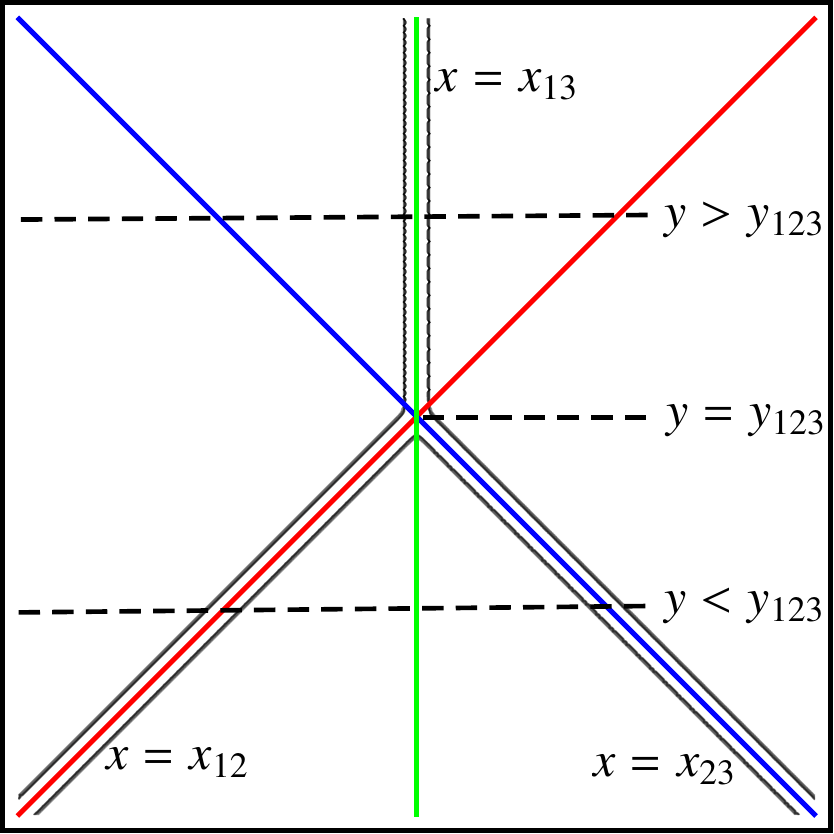} 
\parbox{15cm}{
\caption{To the left is the poset $B(3,2)$ and, below it, its visible part $T(3,2)$. Here the vertices 
are labelled by $Q(\emptyset)$ and $Q(\{[3]\})$, respectively $R(\emptyset)$ and $R(\{[3]\})$. 
The latter data translate into the soliton solution with $M=2$, as indicated 
in the plot on the right hand side (also see \cite{DM_DMH11KPT}). At a thin line, two phases coincide. 
Only the thickened parts are visible. 
We note that $Q(\emptyset)$ and $Q(\{[3]\})$ are linear orders in this particular example, 
hence elements of $A(3,2)$ and, by a general result, maximal chains of $B(3,1)$, see Figure~\ref{DM_fig:B31,T31}. \label{DM_fig:Miles_xy_B32_T32} }
}
\end{center} 
\end{figure} 

\begin{figure}[H] 
\begin{center}
\begin{minipage}[t]{3.8cm}
\begin{minipage}{3.8cm}
\includegraphics[scale=.6]{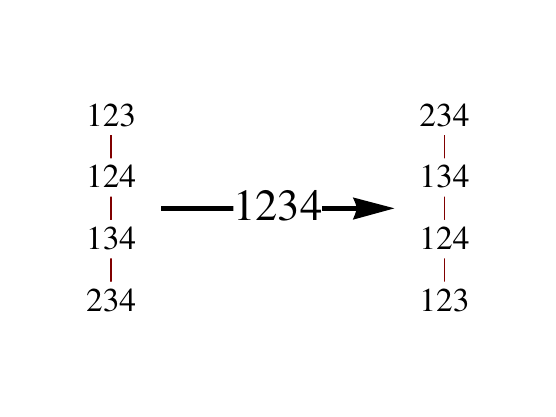} 
\end{minipage}
\begin{minipage}{3.7cm}
\vspace{-1.cm}
\hspace*{.01cm}
\includegraphics[scale=.6]{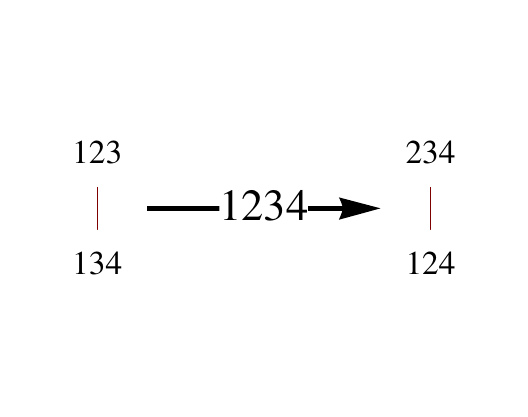}
\end{minipage}
\end{minipage}
\begin{minipage}[t]{8cm}
\vspace{-.7cm}
\includegraphics[scale=.32]{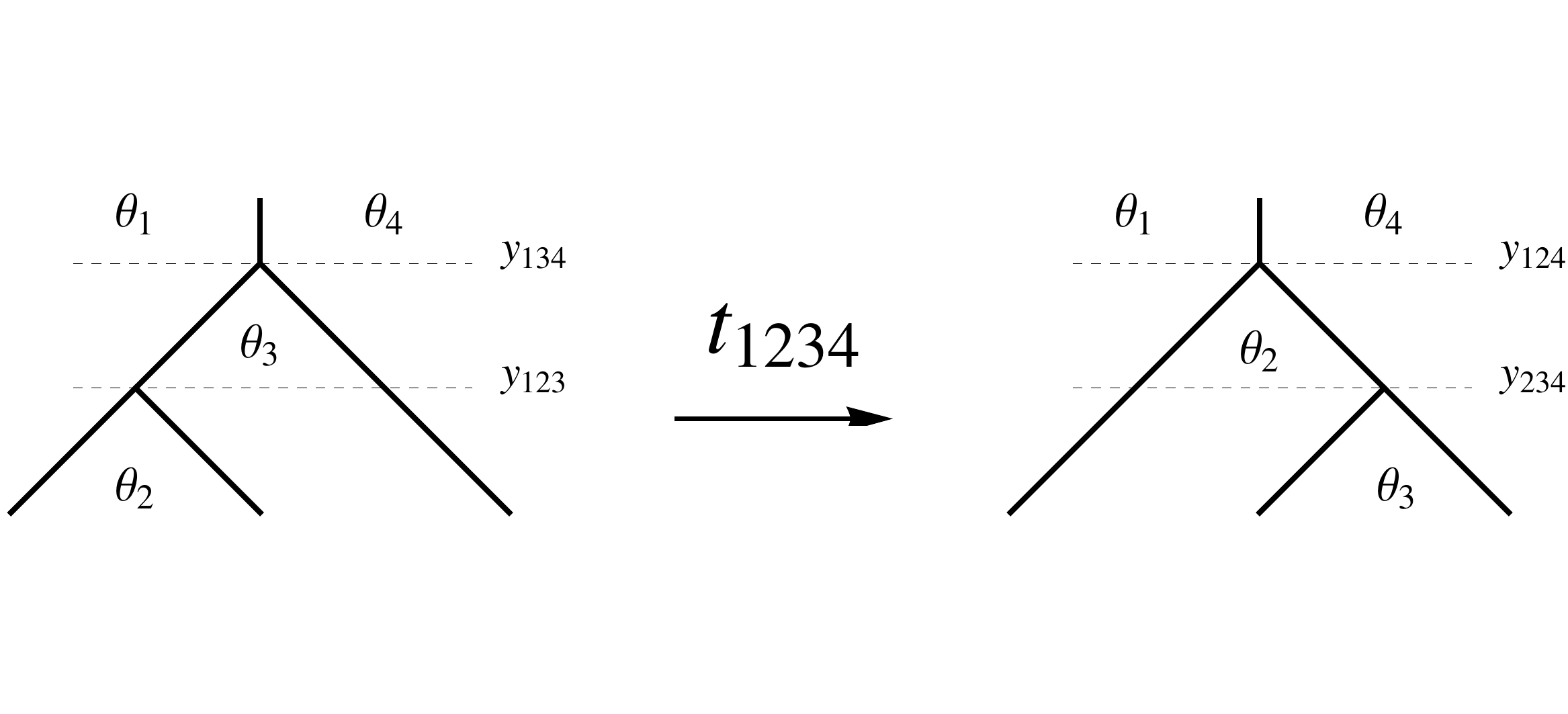}
\end{minipage}
\parbox{15cm}{
\caption{$B(4,3)$ with vertices the posets $Q(\emptyset)$ and $Q(\{[4]\})$. 
The diagram below it is $T(4,3)$, with vertices $R(\emptyset)$ and $R(\{[4]\})$. 
To the right is the translation of $T(4,3)$ into a chain of two rooted binary trees
related by a tree rotation. This chain is the Tamari lattice\index{Tamari!lattice} 
$\mathbb{T}_2$ (in terms of rooted binary trees). \label{DM_fig:B431,T43} }
}
\end{center}
\end{figure} 

The Bruhat order\index{order!higher Bruhat} $B(N,N-2)$ consists of two maximal chains, 
\bez
  \bullet \stackrel{[N-1]}{\longrightarrow} \bullet 
            \stackrel{[N] \setminus \{N-1\}}{\longrightarrow} \bullet
     \; \cdots \; \bullet \stackrel{[N] \setminus \{1\}}{\longrightarrow} \bullet \, , \qquad
   \bullet \stackrel{[N]\setminus \{1\}}{\longrightarrow} \bullet 
            \stackrel{[N] \setminus \{2\}}{\longrightarrow} \bullet
     \; \cdots \; \bullet \stackrel{[N-1]}{\longrightarrow} \bullet \, ,
\eez
in which $P([N])$ appears in lexicographic, respectively reverse lexicographic 
 order\index{order!lexicographic}.
In the first chain we have to eliminate the edges corresponding to all sets $[N] \setminus \{k\}$ 
with $k \in [N]_< = \{N-1,N-3, \ldots , 1 + (N \, \mathrm{mod} \, 2)\}$, in the second chain 
those corresponding to such sets with  
$k \in [N]_> = \{N,N-2, \ldots , 2 - (N \, \mathrm{mod} \, 2)\}$. 
This results in the two reduced chains
\bez
  && \bullet \stackrel{[N-1]}{\longrightarrow} \bullet 
            \stackrel{[N] \setminus \{N-2\}}{\longrightarrow} \bullet
     \; \cdots \; \bullet \stackrel{[N] \setminus \{2 - (N \, \mathrm{mod} \, 2)\}}{\longrightarrow}
     \bullet  \\
  && \bullet \stackrel{[N]\setminus \{1+(N \, \mathrm{mod} \, 2)\}}{\longrightarrow} \bullet 
       \; \cdots \; \bullet
            \stackrel{[N] \setminus \{N-3\}}{\longrightarrow} \bullet
      \stackrel{[N] \setminus \{N-1\}}{\longrightarrow} \bullet \, ,
\eez
which form $T(N,N-2)$. See Figure~\ref{DM_fig:B42,T42} for the case $N=4$. 
\begin{SCfigure}[1.2][hbtp]
\includegraphics[scale=.6]{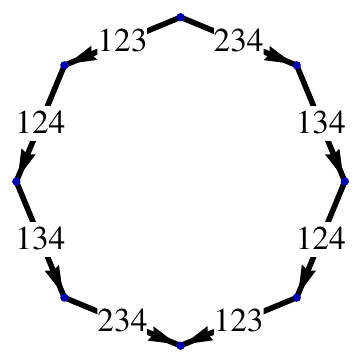} 
\hspace{1.cm}
\includegraphics[scale=.65]{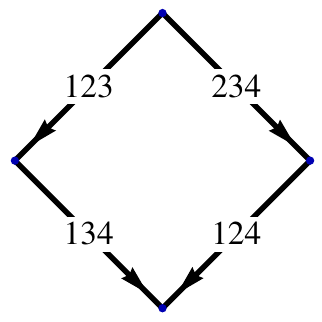} 
\caption{$B(4,2)$ and $T(4,2)$.  } 
\label{DM_fig:B42,T42}
\end{SCfigure} 
The two maximal chains of $B(4,2)$ with the $Q$-posets as vertices have already been 
displayed in Figure~\ref{DM_fig:B42Qs}. Elimination of non-visible elements yields 
the chains of $T(4,2)$ in Figure~\ref{DM_fig:T42chains}.

\begin{figure}
\begin{center}
\includegraphics[scale=.9]{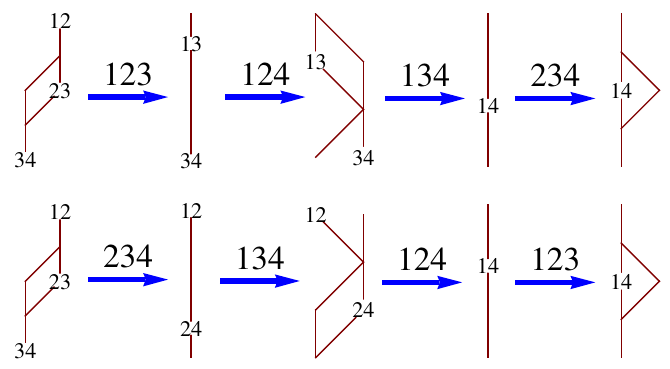} 
\hspace{1cm}
\includegraphics[scale=.9]{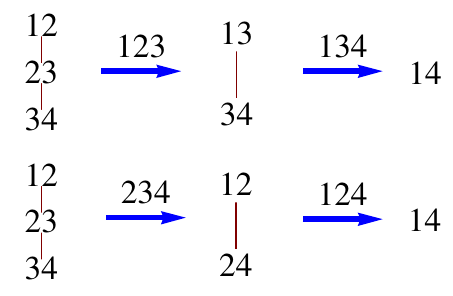} 
\parbox{15cm}{
\caption{The two maximal chains of $T(4,2)$ are displayed on the right hand side.
On the left hand side, in order to illustrate Proposition~\ref{DM_prop:R(U)=R(U')}, 
we show an intermediate elimination step applied to the $B(4,2)$-chains of $Q$-posets 
in Figure~\ref{DM_fig:B42Qs}. Here we still kept the edges $124$ and $234$ 
in the upper chain, and $134$, $123$ in the lower chain, which are non-visible in the 
Tamari order\index{order!higher Tamari}. We observe that they indeed connect identical posets. 
\label{DM_fig:T42chains} }
}
\end{center}
  
\end{figure} 

\begin{figure}
\begin{center}
\includegraphics[scale=.4]{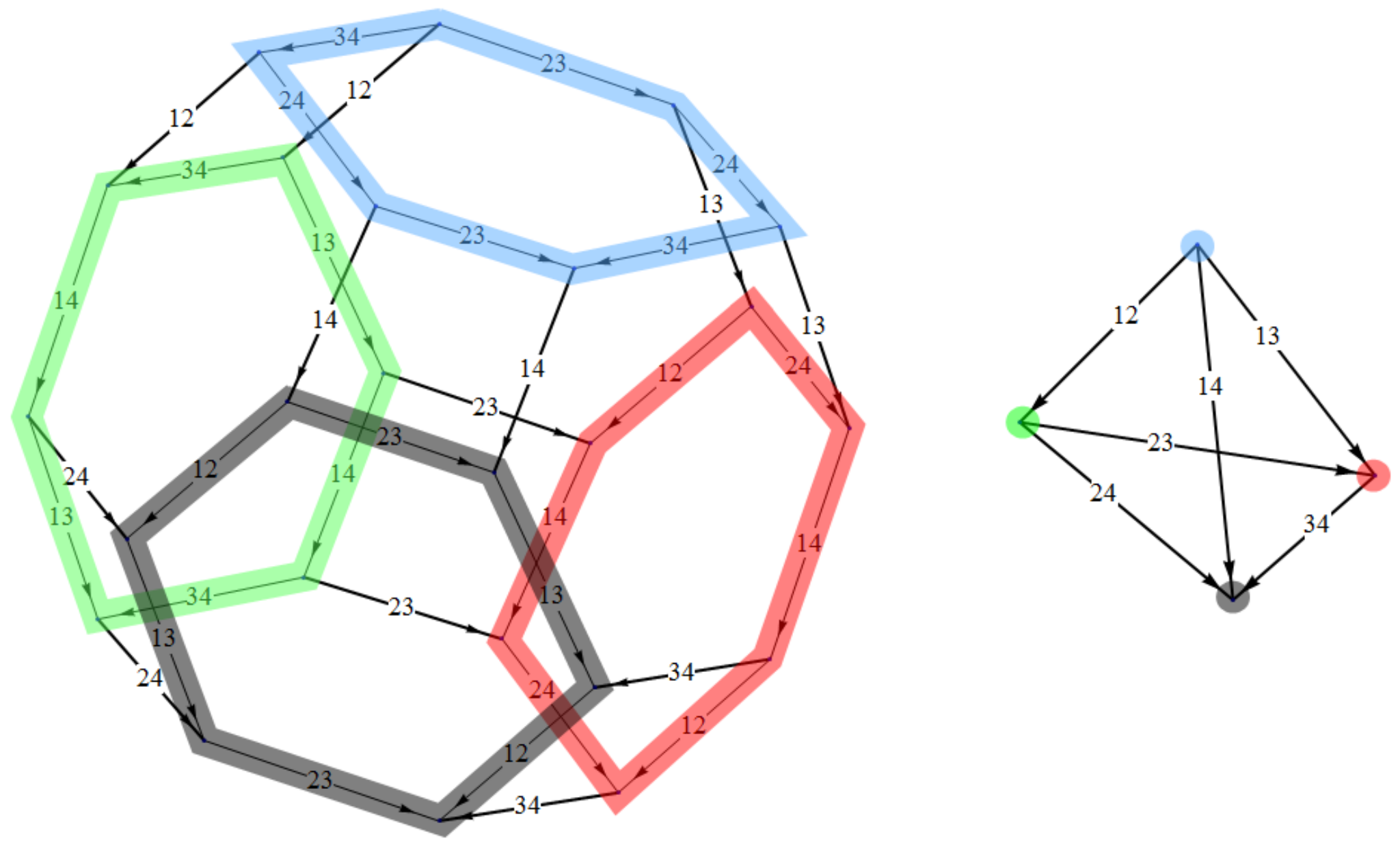}
\parbox{15cm}{
\caption{The left figure shows $B(4,1)$. Non-visible edges connecting vertices that 
are mapped to the same vertex of $T(4,1)$ (tetrahedral poset) 
are marked with the same color. \label{DM_fig:B41,T41} }
}
\end{center}  
\end{figure}

By determining the linear extensions $K_1 \rightarrow K_2  \rightarrow \cdots$ 
of the posets in Figure~\ref{DM_fig:B42Qs}, we obtain maximal chains 
$\bullet \stackrel{K_1}{\longrightarrow} \bullet \stackrel{K_2}{\longrightarrow} \cdots$ 
of $B(4,1)$. In this way we construct $B(4,1)$ and then obtain $T(4,1)$ from it by elimination 
of non-visible edges, see Figure~\ref{DM_fig:B41,T41}. 
Figure~\ref{DM_fig:B52,T52} shows the corresponding construction of $T(5,2)$ from $B(5,2)$. 

\begin{figure}
\begin{center}
\includegraphics[scale=.42]{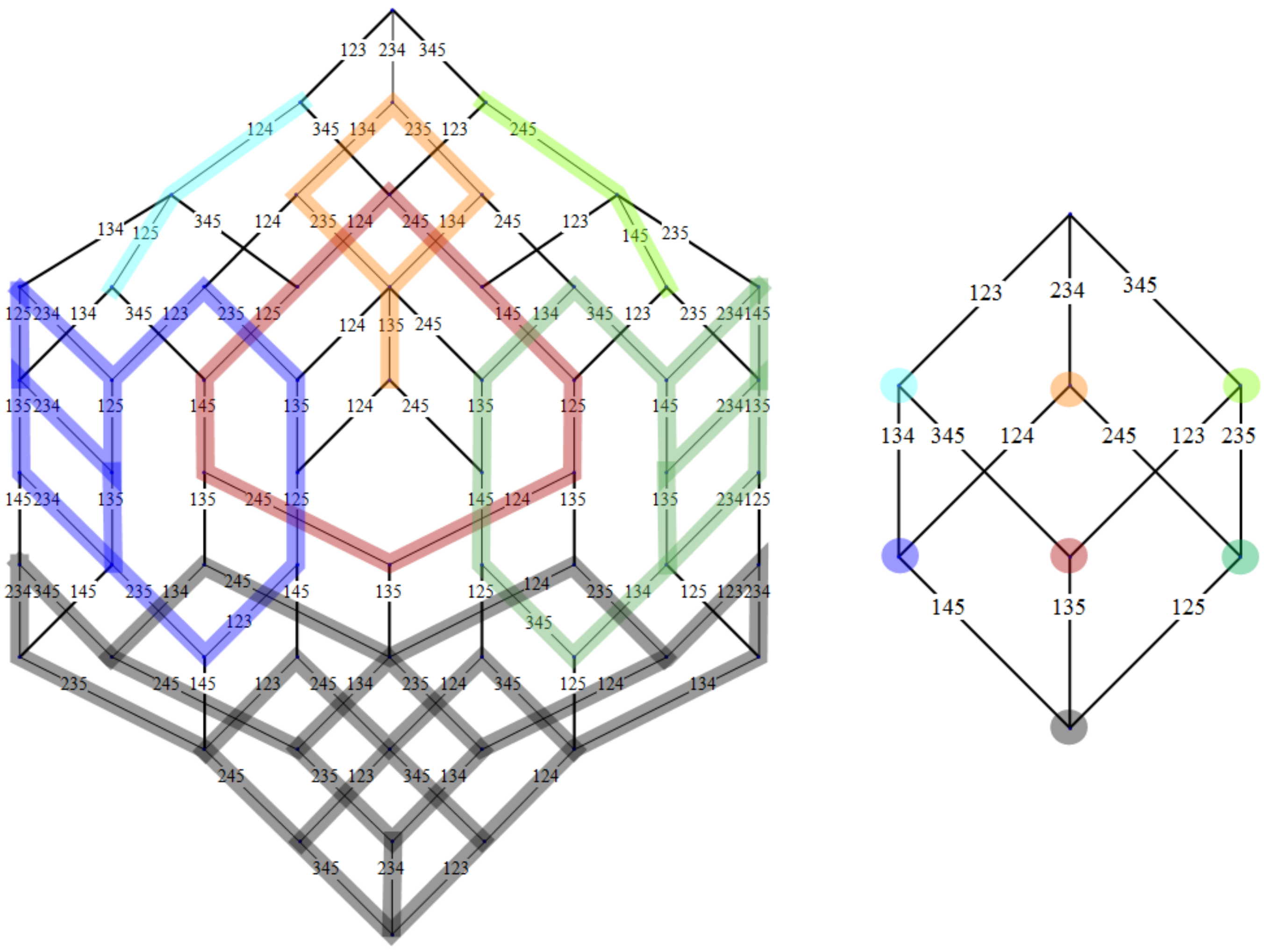} 
\parbox{15cm}{
\caption{$B(5,2)$ and $T(5,2)$. Vertices of $B(5,2)$ connected by edges marked with the same color 
are mapped to the same vertex of $T(5,2)$. \label{DM_fig:B52,T52} }
}
\end{center}
\end{figure}

\subsubsection{An $(n+1)$-gonal equivalence relation}
\label{DM_ssec:polygonal_rel}
The vertices of $B(N,n)$ can be described by the $Q$-posets, and on this level 
we defined the transition to $T(N,n)$. 
But the vertices of $B(N,n)$ are equivalently given by consistent sets (which was 
in fact our original definition). Along a maximal chain, moving from a vertex to 
the next means increasing the consistent set associated with the first vertex 
by inclusion of a new set which we use to label the edge between the two 
vertices. 
The transition from a Bruhat\index{order!higher Bruhat} to a Tamari order\index{order!higher Tamari} 
means elimination of some of these sets. 
Whereas every maximal chain of $B(N,n)$ ends in the same set, this is not 
so for different Tamari chains because of different eliminations along different  
chains, see Figure~\ref{DM_fig:polygon_rule}.
\begin{SCfigure}[1.6][hbtp]
\hspace{1.cm}
\includegraphics[scale=.5]{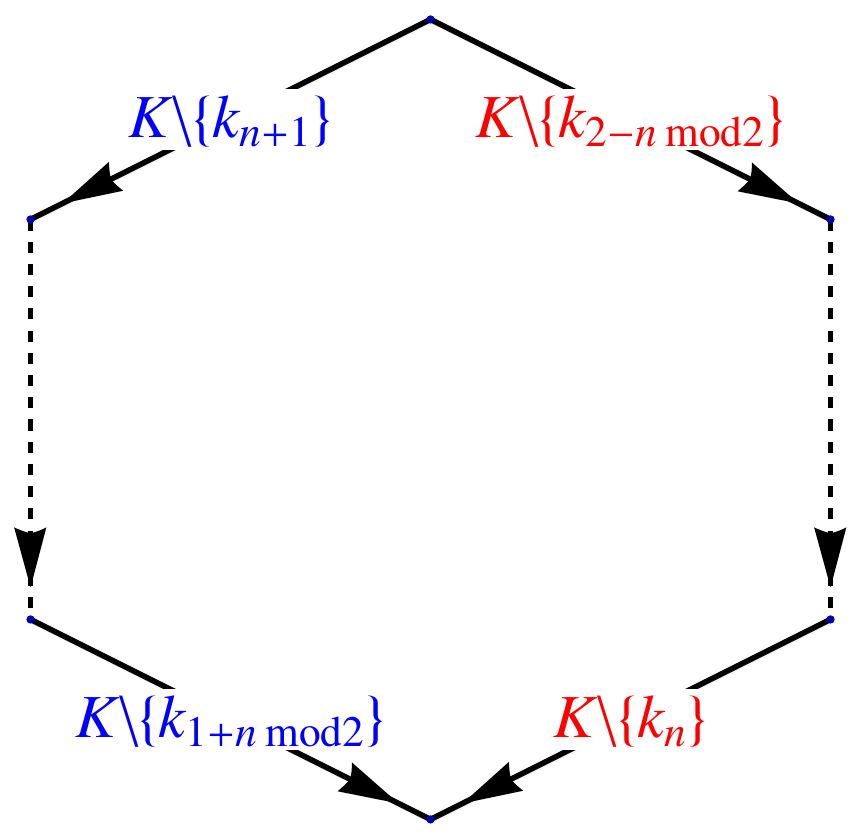} 
\caption{For $K \in {[N] \choose n+1}$, the two chains of the diagram start in 
the same vertex and end in the same vertex, which includes the union of all sets 
associated with the edges of the left chain, but 
also the union of all sets associated with the edges of the right chain. 
The two resulting sets have to be identified, which leads to an $(n+1)$-gonal 
equivalence relation. If $K = \{[N]\}$, i.e. $n=N-1$, the chains are the 
two maximal chains of $T(N,N-2)$. }
\label{DM_fig:polygon_rule}  
\end{SCfigure} 
If $U$ is the set that corresponds to the first vertex from which the two chains in 
 Figure~\ref{DM_fig:polygon_rule} descend, the first ends in 
$U \cup (K \setminus \{k_{n+1}\}) \cup \cdots \cup (K \setminus \{k_{1+n \, \mathrm{mod} \, 2} \})$,
the second in  
$U \cup (K \setminus \{k_{2-n \, \mathrm{mod} \, 2} \}) \cup \cdots \cup (K \setminus \{k_n\})$. 
The two resulting sets have to be identified, since the final vertex is the same. 
This requires the \emph{$(n+1)$-gonal equivalence relation}
\bez
   U \cup \{ K \setminus \{k_{n+1}\}, \ldots, K \setminus \{k_{1+n \, \mathrm{mod} \, 2} \} \} 
  \sim U \cup \{ K \setminus \{k_n\}, \ldots, K \setminus \{k_{2-n \, \mathrm{mod} \, 2} \} \} \, ,
\eez 
for a Tamari order\index{order!higher Tamari}, and motivates the algebraic structure considered in 
Section~\ref{DM_sec:polygonal_rels}.

\section{KP line soliton evolutions in the case $M=5$}
For $M=5$ (i.e., $N=6$) the $\tau$-function of a tree-shaped line soliton solution is given by 
\bez
  \tau = e^{\theta_1} + \cdots + e^{\theta_6} \, , \qquad
  \theta_i = p_i \, x + p_i^2 \, y + p_i^3 \, t + p_i^4 \, t^{(4)} + p_i^5 \, t^{(5)} + c_i 
  \, ,
\eez
with real constants $p_1< \cdots <p_6$ and $c_i$. 
We have $\Omega = [6] = \{1,2,3,4,5,6\}$, and $\bullet \stackrel{[6]}{\longrightarrow} \bullet$ 
represents $B(6,5)$ and also $T(6,5)$. The evolution of the soliton corresponds to 
the left vertex if $t^{(5)} < t^{(5)}_{123456}$, and to the right vertex if 
$t^{(5)} > t^{(5)}_{123456}$. Associated with the two vertices are maximal chains of 
$B(6,4)$, from which $T(6,4)$ is obtained by elimination of non-visible events, 
see Figure~\ref{DM_fig:B64,T64}. 
%\vspace{-.8cm}
\begin{SCfigure}[1.][hbtp] 
\hspace{.5cm}
\includegraphics[scale=.38]{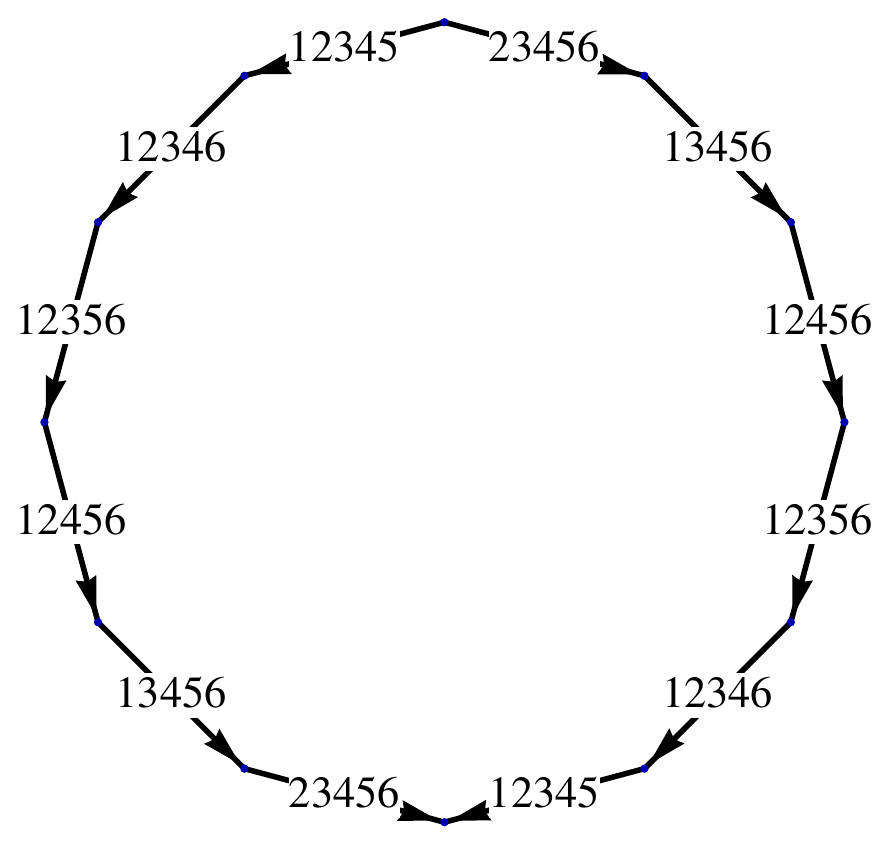} 
\hspace{.1cm}
\includegraphics[scale=.76]{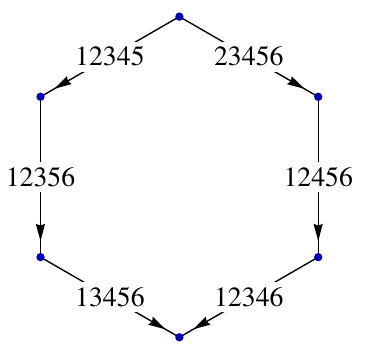} 
\caption{$B(6,4)$ and $T(6,4)$. Here e.g. $12356$ (which stands for $\{1,2,3,5,6\}$) 
translates to the value $t^{(4)}_{12356}$ of the parameter $t^{(4)}$. }
\label{DM_fig:B64,T64}  
\end{SCfigure} 
If $t^{(5)} < t^{(5)}_{123456}$, the left chain of $B(6,4)$ applies, which means
\bez
    t^{(4)}_{12345} < {\color{green}t^{(4)}_{12346}} < t^{(4)}_{12356} 
    < {\color{green}t^{(4)}_{12456}} < t^{(4)}_{13456} < {\color{green}t^{(4)}_{23456}} \, , 
\eez
where the second, fourth and sixth value corresponds to a non-visible event.  
If $t^{(5)} > t^{(5)}_{123456}$, the right chain of $B(6,4)$ applies, hence
\bez
    t^{(4)}_{23456} < {\color{green}t^{(4)}_{13456}} < t^{(4)}_{12456} 
    < {\color{green}t^{(4)}_{12356}} < t^{(4)}_{12346} < {\color{green}t^{(4)}_{12345}} \, ,
\eez
where the second, fourth and sixth value is non-visible. 
If $t^{(5)} = t^{(5)}_{123456}$, all the values $t^{(4)}_{ijklm}$ coincide. 

\begin{figure}
\begin{center}
\includegraphics[scale=1.2]{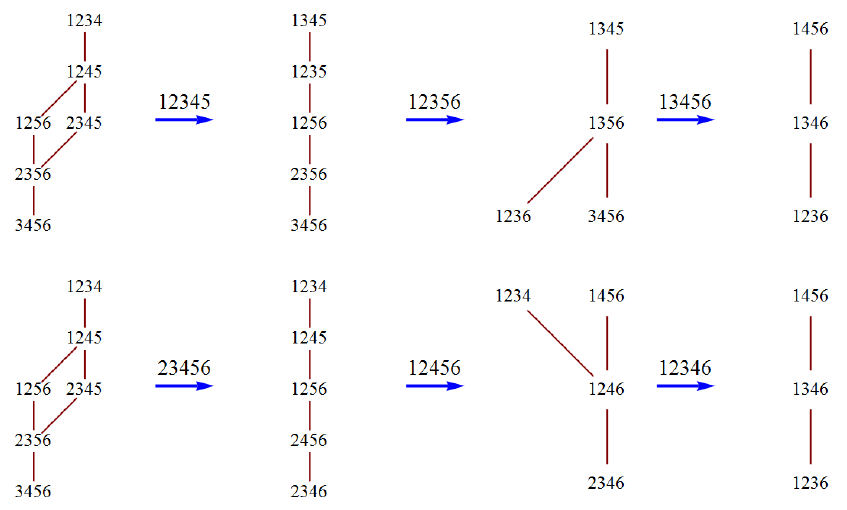} 
\parbox{15cm}{
\caption{The two maximal chains of $T(6,4)$. The vertices are the $R$-posets obtained 
from the $Q$-posets associated with the vertices of $B(6,4)$. 
The upper chain applies if $t^{(5)} < t^{(5)}_{123456}$, the lower if 
$t^{(5)} > t^{(5)}_{123456}$. \label{DM_fig:T64chains} } 
}
\end{center} 
\end{figure} 

The vertices of $T(6,4)$ are the $R$-posets in Figure~\ref{DM_fig:T64chains}. 
The linear extensions of a poset determine the possible 
orders of critical values 
of time $t = t^{(3)}$. For example, if $t^{(5)} > t^{(5)}_{123456}$, the lower 
horizontal chain in Figure~\ref{DM_fig:T64chains} applies. If furthermore 
$t^{(4)}_{12456} < t^{(4)} < t^{(4)}_{12346}$ (third poset), then we have either 
$t_{1234} < t_{1456} <  t_{1246} < t_{2346}$ or $t_{1456} < t_{1234} <  t_{1246} < t_{2346}$ 
(since the poset has two linear extensions). In order to decide which of 
these orders is realized by the soliton, further conditions on the parameters $p_i$ 
(not the $c_i$) are required (see \cite{DM_DMH11KPT}). 
 From the linear extensions of the $R$-posets, we obtain (the maximal chains of) $T(6,3)$, 
which is the Tamari lattice\index{Tamari!lattice} $\mathbb{T}_4$. Figure~\ref{DM_fig:T4polytope} 
displays it as a (Tamari-Stasheff) polytope\index{associahedron}.
In order to identify the vertices of $T(6,3)$, we take the pentagonal equivalence 
\bez
 \{ \{i,j,k,l\}, \{i,j,l,m\}, \{j,k,l,m\} \} \sim \{ \{i,j,k,m\}, \{i,k,l,m\} \}
\eez
(where $i<j<k<l<m$) into account (see Section~\ref{DM_ssec:polygonal_rel}). 
The binary trees labelling its vertices in Figure~\ref{DM_fig:T4} are 
obtained from the $R$-posets, listed in Figure~\ref{DM_fig:T4vertices_posets}.

\begin{SCfigure}[1.2][hbtp]
\hspace{1.cm}
\includegraphics[scale=.4]{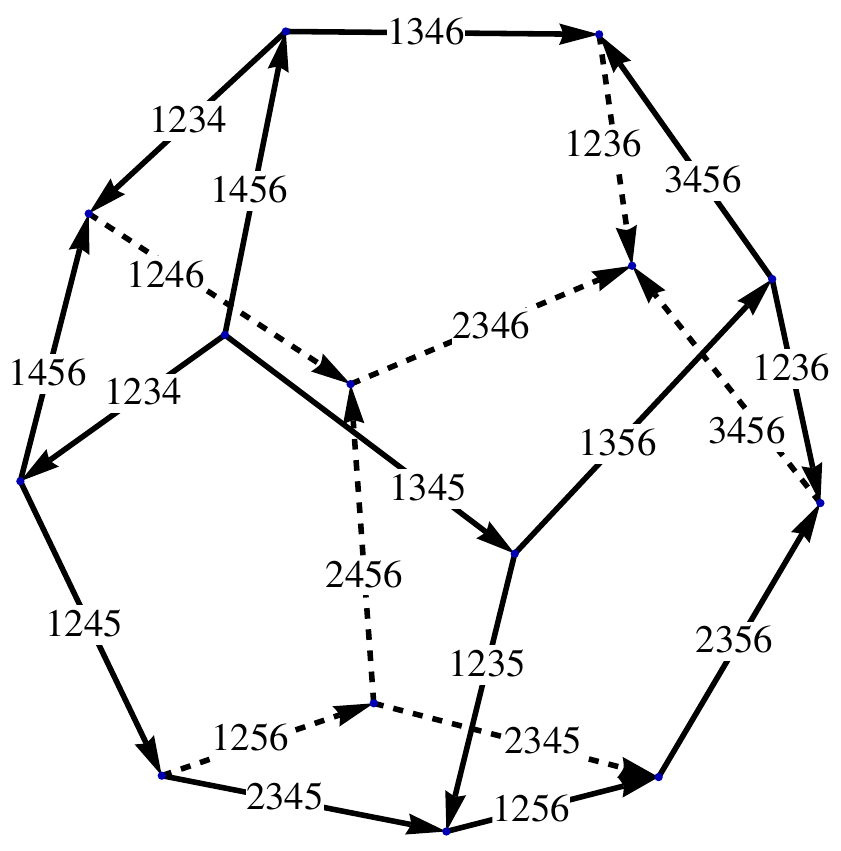}
\caption{The Tamari lattice\index{Tamari!lattice} $\mathbb{T}_4 = T(6,3)$ as a Tamari-Stasheff polytope\index{associahedron}. 
Such a representation first appeared in Tamari's thesis\index{Tamari!thesis} in 1951 \cite{DM_Tamari1951thesis}. }
\label{DM_fig:T4polytope}  
\end{SCfigure} 

\begin{figure}
\begin{center}
\includegraphics[scale=.7]{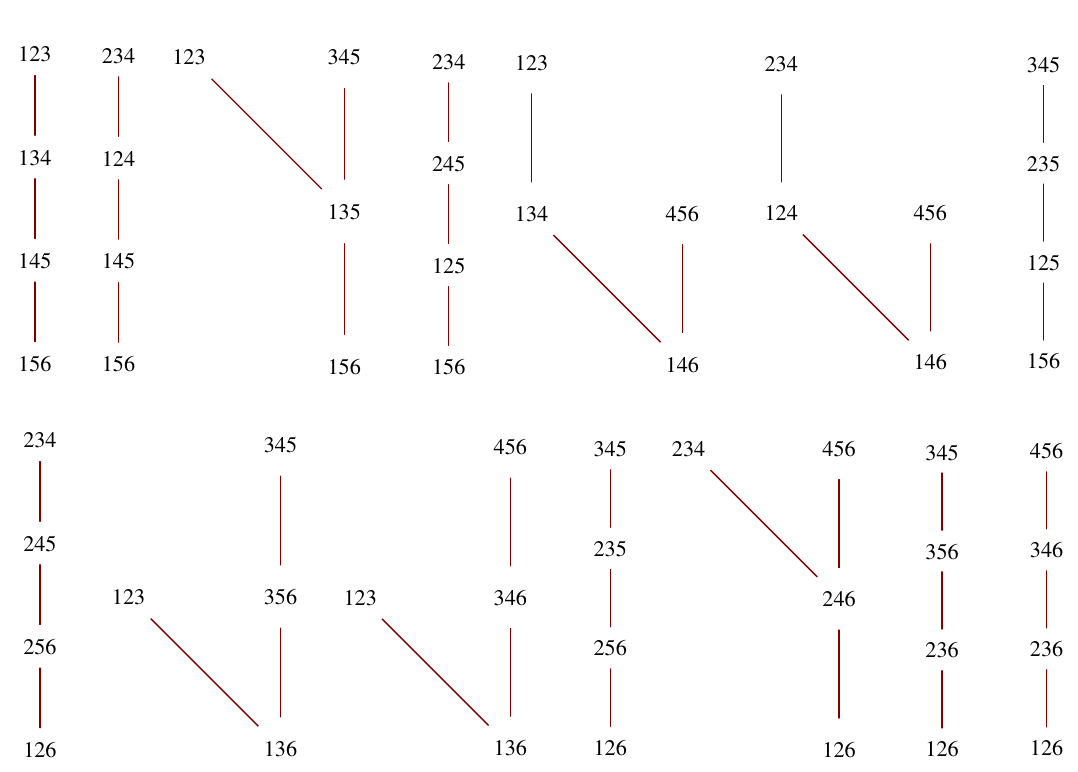} 
\parbox{15cm}{
\caption{The fourteen $R$-posets that label the vertices of $T(6,3)=\mathbb{T}_4$. 
They translate into the trees labelling the vertices of the Tamari 
lattice\index{Tamari!lattice} $\mathbb{T}_4$ 
in Figure~\ref{DM_fig:T4}. Each poset determines more directly a triangulation of a hexagon, 
the vertices of which are numbered (anticlockwise) by $1,2,\ldots,6$. 
A triple $ijk$ then specifies a triangle. \label{DM_fig:T4vertices_posets}  }
}
\end{center}  
\end{figure} 

Figure~\ref{DM_fig:T4evolution_example} shows an example of a line soliton evolution 
and the caption identifies the corresponding maximal chain of $\mathbb{T}_4$.
\begin{figure} 
\begin{center}
\includegraphics[scale=.57]{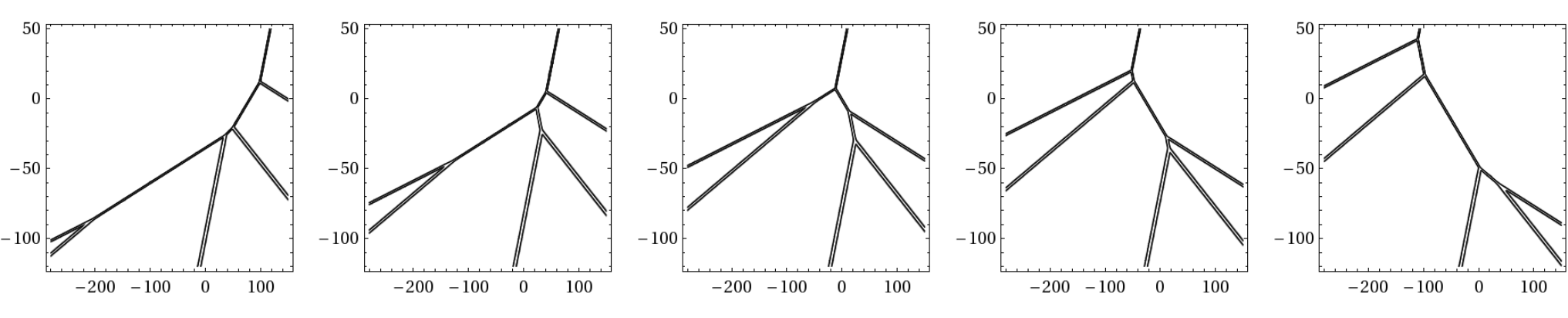} 
\parbox{15cm}{
\caption{Plots of an $M=5$ soliton in the $xy$-plane at successive values of 
time. Here we have chosen the parameters such that $t^{(5)}<t^{(5)}_{123456}$ and 
$t^{(4)}_{12356} < t^{(4)} < t^{(4)}_{13456}$. The evolution corresponds to 
$\bullet \stackrel{1345}{\longrightarrow} \bullet \stackrel{1356}{\longrightarrow} \bullet 
\stackrel{1236}{\longrightarrow} \bullet \stackrel{3456}{\longrightarrow} \bullet$ 
on the Tamari lattice\index{Tamari!lattice} $T(6,3)=\mathbb{T}_4$ in Figure~\ref{DM_fig:T4}. 
It is a linear extension of the third poset of the first horizontal chain 
in Figure~\ref{DM_fig:T64chains}. \label{DM_fig:T4evolution_example}  }
}
\end{center} 
\end{figure} 

In the next steps, we obtain $T(6,2)$ and $T(6,1)$, see Figure~\ref{DM_fig:T(6,2),T(6,1)}.
\begin{figure}
\begin{center}
\includegraphics[scale=.55]{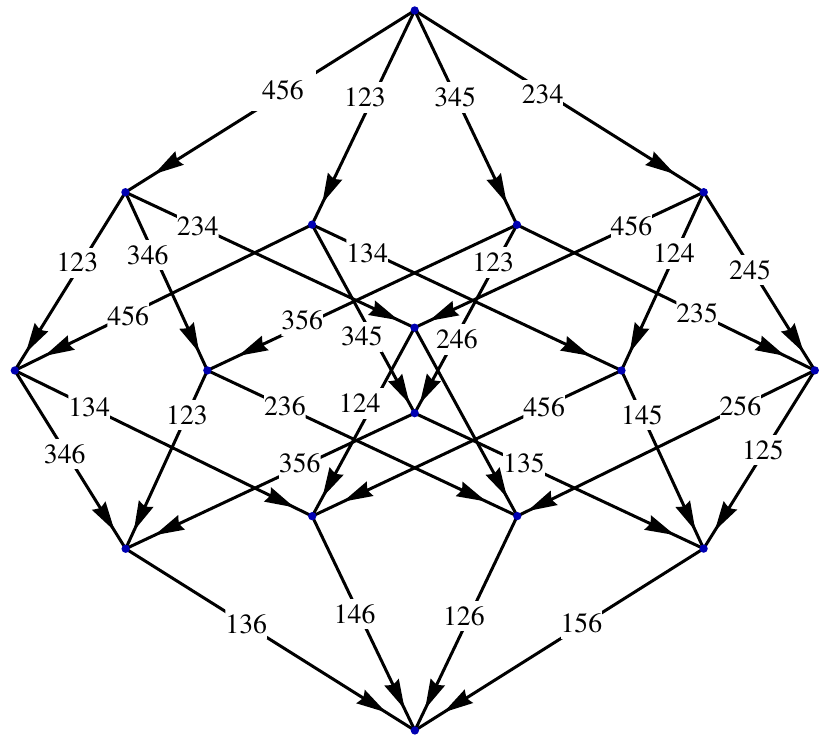} 
\hspace{1.cm}
\includegraphics[scale=.52]{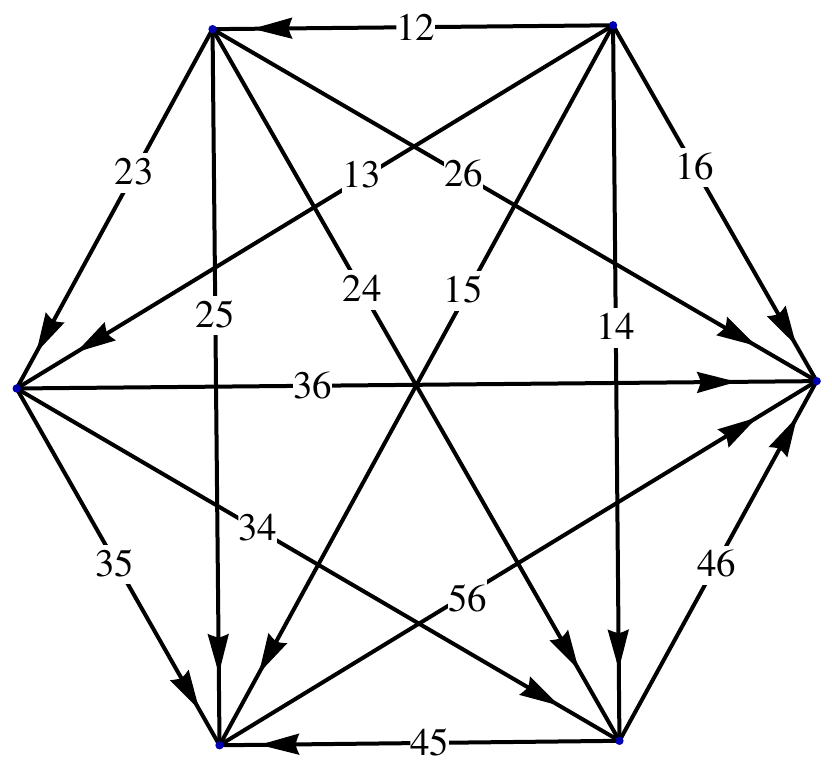} 
\parbox{15cm}{
\caption{The Tamari order\index{order!higher Tamari} $T(6,2)$, which forms a cube in four dimensions, and 
$T(6,1)$, which is a $5$-simplex. \label{DM_fig:T(6,2),T(6,1)}  }
}
\end{center}
\end{figure} 
Finally, $T(6,0)$ is given by the Hasse diagram\index{Hasse diagram} with two 
vertices, the upper connected 
with the lower by six edges (labelled by $1,2,\ldots,6$).

\section{Some insights into the case $M > 5$}
The subclass of tree-shaped line soliton solutions with $M=6$, hence $N=7$ and $\Omega = [N] = \{1,2,3,4,5,6,7\}$, is given by 
\bez
  \tau = e^{\theta_1} + \cdots + e^{\theta_7} \, , \qquad
  \theta_i = p_i \, x + p_i^2 \, y + p_i^3 \, t + p_i^4 \, t^{(4)} + p_i^5 \, t^{(5)} 
             + p_i^6 \, t^{(6)} + c_i \; .
\eez
$T(7,5)$ is the heptagon in Figure~\ref{DM_fig:T75}. 
The left chain is realized if $t^{(6)} < t^{(6)}_\Omega$, the right chain if 
$t^{(6)} > t^{(6)}_\Omega$. 
Each element (vertex) is an $R$-poset, they are displayed in Figure~\ref{DM_fig:T75chains}. 
\begin{SCfigure}[1.4][hbtp] 
\includegraphics[scale=.6]{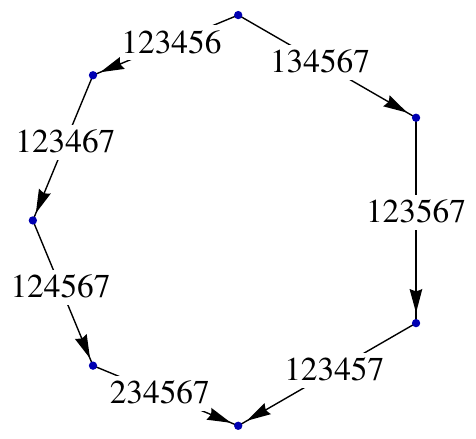} 
\hspace*{2.5cm}
\caption{The Tamari order\index{order!higher Tamari} $T(7,5)$. }
\label{DM_fig:T75} 
\end{SCfigure} 

\begin{figure}
\begin{center}
\includegraphics[scale=.9]{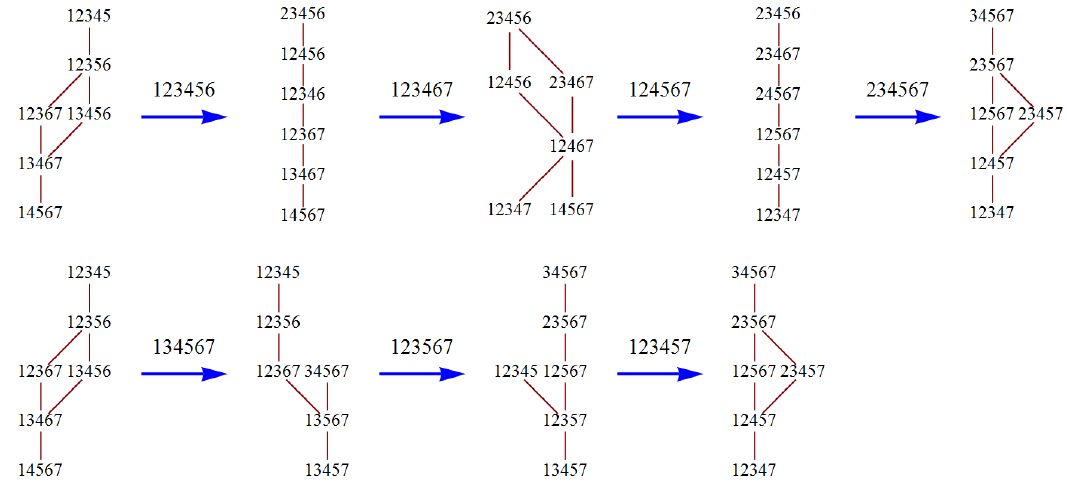} 
\parbox{15cm}{
\caption{The two maximal chains of $T(7,5)$, with vertices resolved into $R$-posets. 
 \label{DM_fig:T75chains} }
}
\end{center}  
\end{figure}  

 From these $R$-posets, we obtain in turn the maximal chains of $T(7,4)$, hence we can 
construct $T(7,4)$ by putting all these chains together (joining a minimal and a maximal 
element). In this way we recover a poset that first appeared in \cite{DM_Edel+Rein96} (Figure~4 
therein). Using {\sc Mathematica} \cite{DM_Mathematica}, we obtained a pseudo-realization 
as a polytope, see Figure~\ref{DM_fig:T74,T85polytopes}.

\begin{SCfigure}[1.][hbtp]
\hspace{.4cm}
\includegraphics[scale=.17]{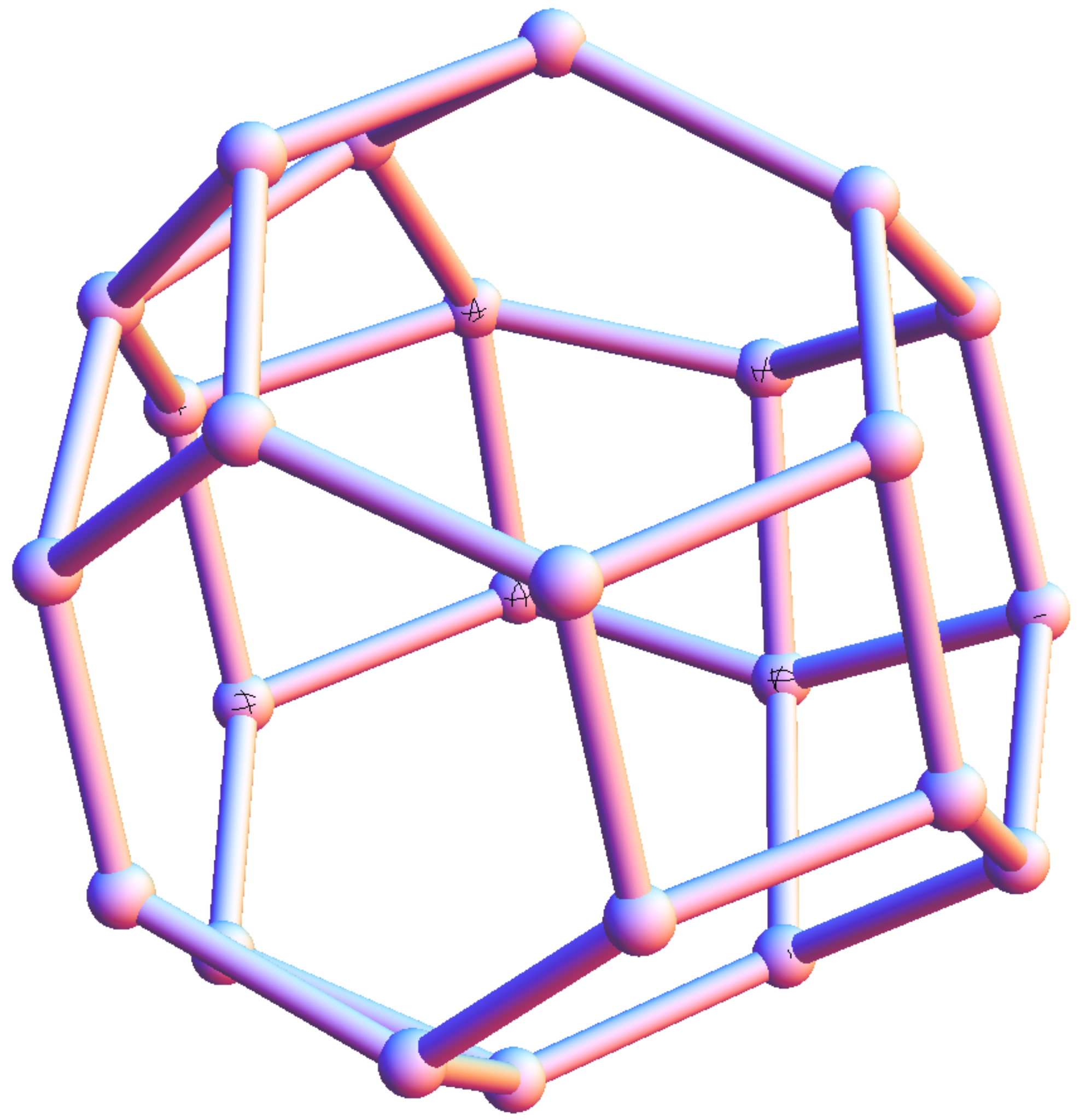} 
\hspace{.1cm}
\includegraphics[scale=.18]{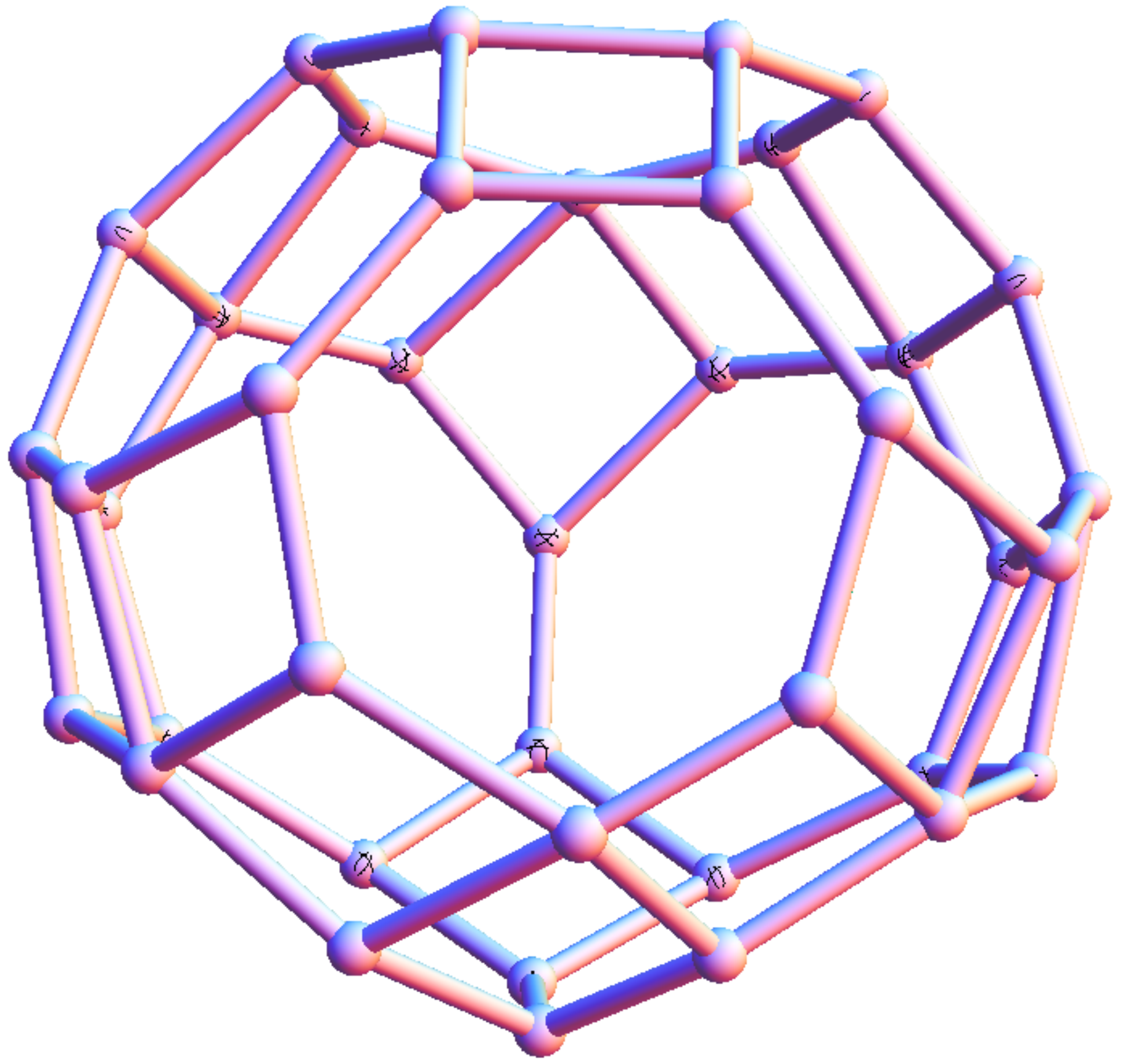} 
\caption{Polytopes on which the higher Tamari orders\index{order!higher Tamari} $T(7,4)$ and $T(8,5)$ live. 
It should be noticed, however, that not all faces are \emph{regular} or \emph{flat} 
quadrangles or hexagons, respectively heptagons 
(as also in Figure~\ref{DM_fig:T4polytope} with pentagons, cf. \cite{DM_Loday11}). }
\label{DM_fig:T74,T85polytopes}  
\end{SCfigure}  
In the next step, we obtain the Tamari lattice\index{Tamari!lattice} $\mathbb{T}_5 = T(7,3)$, 
which can be realized as the 4-dimensional associahedron\index{associahedron}. 
Its maximal chains classify the 
possible evolutions of a tree-shaped line soliton with seven phases. 

Figure~\ref{DM_fig:T74,T85polytopes} also shows that $T(8,5)$ is polytopal. 
Figure~\ref{DM_fig:T96,T107polytopes} displays polytope-like representations of  
some other higher Tamari orders\index{order!higher Tamari}. These are \emph{not} polytopes, however. 
\begin{figure}
\begin{center}
\includegraphics[scale=.5]{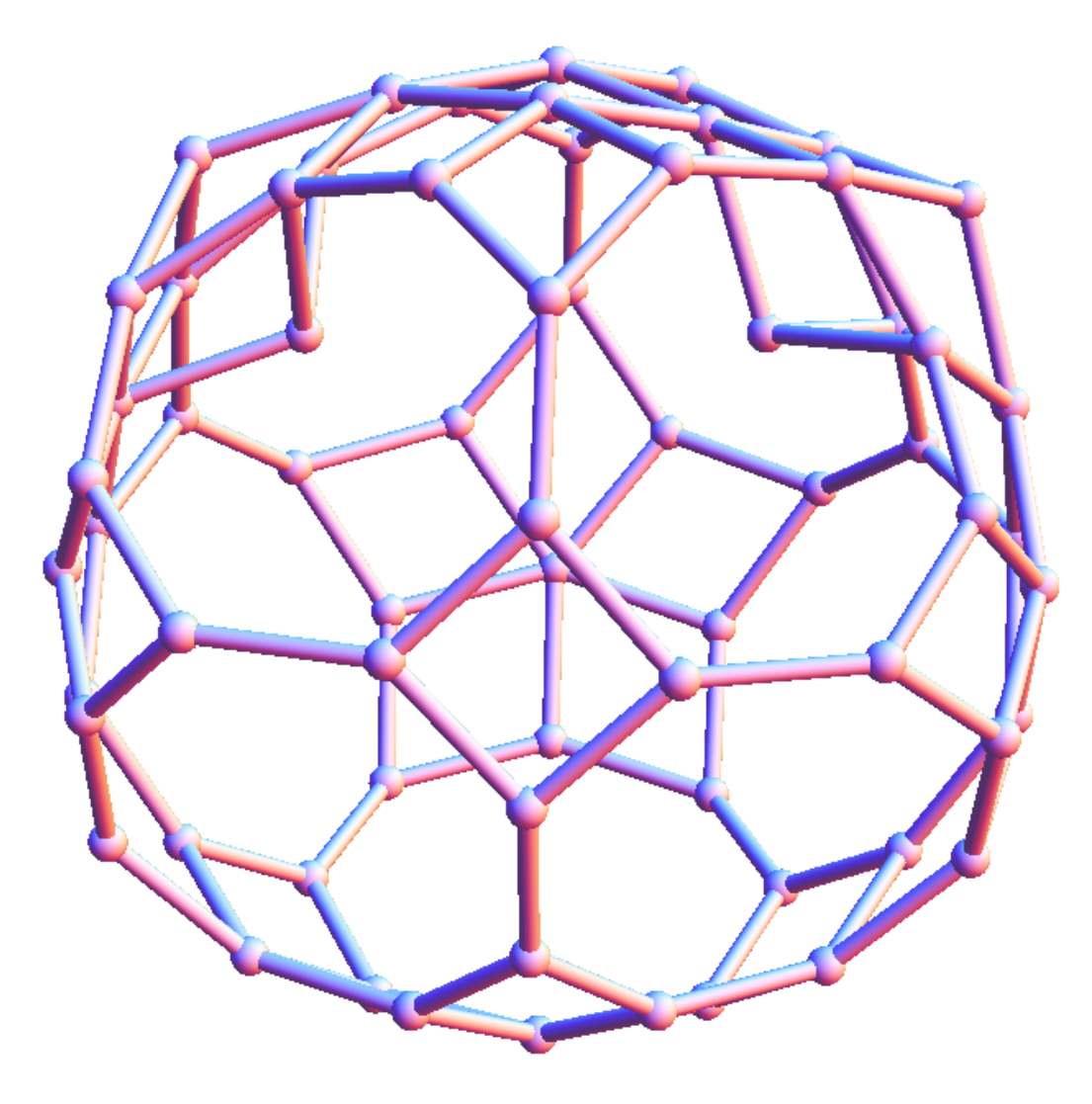}
\hspace{.5cm}
\includegraphics[scale=.5]{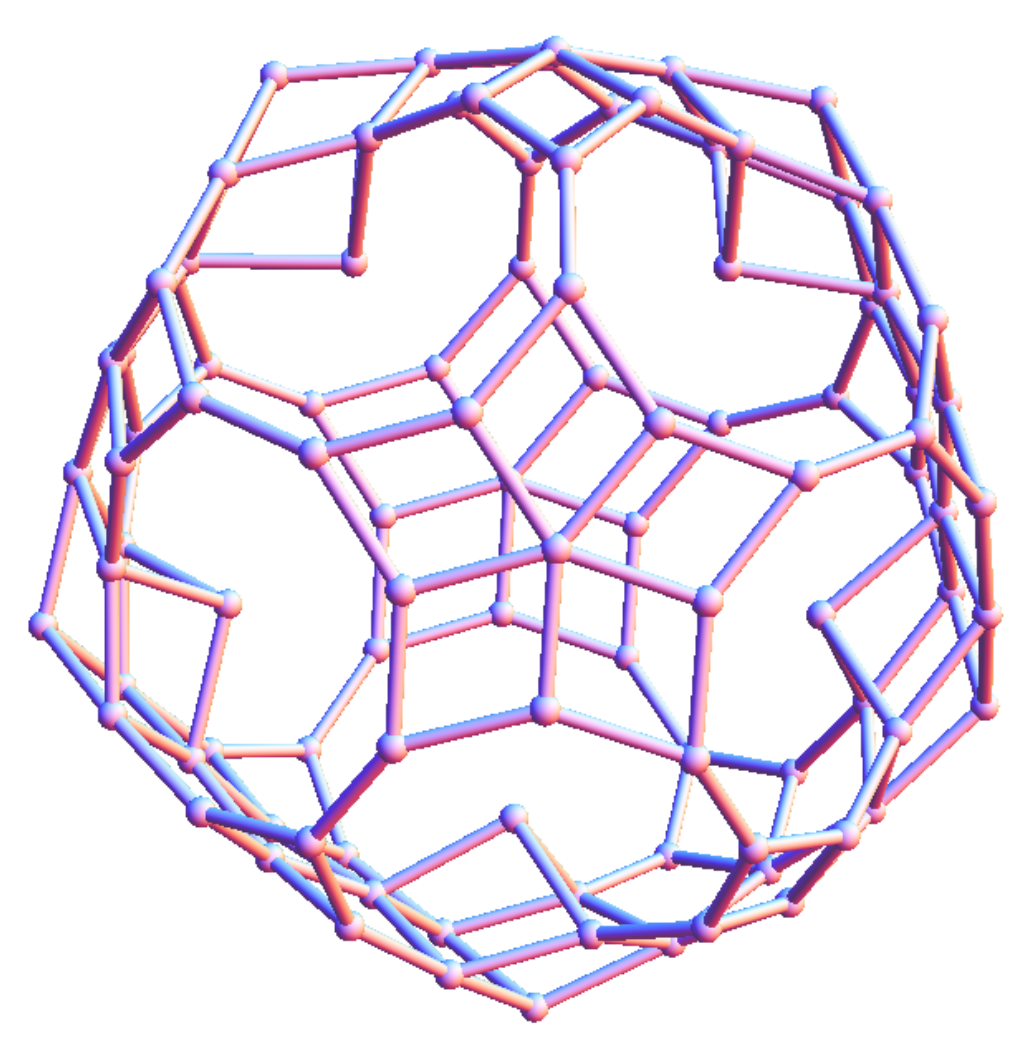}
\parbox{15cm}{
\caption{Polytope-like structures on which the higher Tamari orders\index{order!higher Tamari} $T(9,6)$ 
and $T(10,7)$ live. The existence of `small cubes' (cf. \cite{DM_Felsner+Ziegler01}, 
proof of Observation~5.3) indicates that these are \emph{not} polytopes. There are 
two of them in the left figure, and five in the right. \label{DM_fig:T96,T107polytopes} } 
}
\end{center}
\end{figure}

\section{An algebraic construction of higher Bruhat and Tamari orders}
\label{DM_sec:polygonal_rels}

\subsection{Higher Bruhat orders and simplex equations}
\label{DM_ssec:Bruhat_eqs}
Let $n,N$ be integers with $0<n<N$. Let $\mathcal{B}_{N,n}$ be the monoid 
generated by symbols $R_I$, $I \in { [N] \choose n }$, and 
$R_K$, $K \in { [N] \choose n+1 }$, subject to the following relations.
\begin{itemize}
\item[(a)]\hspace{.2cm} $R_{J} \, R_{J'} = R_{J'} \, R_{J} \;$ 
if $J,J' \in {[N] \choose k}$, $k \in \{n,n+1\}$, are such that $|J \cup J'| > k+1$. 
\item[(b)]\hspace{.2cm} $R_I \, R_K = R_K \, R_I \;\;$ if $I \not\subset K$.
\item[(c)]\hspace{.2cm} For $K =\{k_1,\ldots, k_{n+1}\}$, $k_1 < k_2 < \cdots < k_{n+1}$,  
\be
    R_{ K \setminus \{ k_{n+1} \} } \, R_{ K \setminus \{ k_n \} } \cdots 
    R_{ K \setminus \{ k_1\} } 
  = R_K \; \, R_{ K \setminus \{k_1 \} } \, R_{ K \setminus \{k_2 \} } \cdots 
    R_{ K \setminus \{k_{n+1}\} } \; .   \label{DM_Bruhat_R2}
\ee
\end{itemize}

Recall that, if two neighbors $I_j,I_{j+1} \in {[N] \choose n}$ in a linear order 
$\rho = (I_1,\ldots, I_s) \in A(N,n)$ (where $s={N \choose n}$)  
are not contained in a common packet (or, equivalently, if $|I_j \cup I_{j+1}| > n+1$), 
exchanging them leads to an elementarily equivalent linear order $\rho'$, i.e. $\rho \sim \rho'$. 
As a consequence of the above relations, the map that sends $\rho$ to the monomial 
$R_{I_1} R_{I_2} \cdots R_{I_s}$ induces a correspondence between equivalence classes 
$[\rho]$, and thus elements of the higher Bruhat order $B(N,n)$, and such monomials. 
Relation (c) encodes the order relations of $B(N,n)$. 
It relates a lexicographically ordered product to the reverse lexicographically 
ordered product. In the following, we write 
\bez
     R_{i_1 \ldots i_n} := R_{ \{i_1, \ldots, i_n\}}  \qquad \quad  i_1 < i_2 < \cdots < i_n \; .
\eez

For $n=1$, we have 
\bez
     R_i R_j = R_{ij} \; R_j R_i   \qquad \quad   i<j \, , 
\eez
and, for $i<j<k$, 
\bez
   (R_i R_j) \, R_k = R_{ij} \; R_j (R_i R_k)
                    = R_{ij} \; R_j \, R_{ik} \, R_k R_i
                    = R_{ij} R_{ik} \; (R_j R_k) \, R_i
                    = R_{ij} R_{ik} R_{jk} \; R_k R_j R_i \, ,
\eez
and also
\bez
   R_i \, (R_j R_k) = R_i \, R_{jk} \; R_k R_j
                    = R_{jk} \; (R_i R_k) \, R_j
                    = R_{jk} R_{ik} \; R_k \, (R_i R_j)
                    = R_{jk} R_{ik} R_{ij} \; R_k R_j R_i \; .
\eez
Associativity now leads to the \emph{consistency condition}\index{consistency condition}
\be
    R_{ij} R_{ik} R_{jk} = R_{jk} R_{ik} R_{ij}  \qquad \quad  i<j<k \; .  \label{DM_YB}
\ee
For $N=3$, this is $R_{12} R_{13} R_{23} = R_{23} R_{13} R_{12}$, which has the form of 
the \emph{Yang-Baxter equation}\index{Yang-Baxter equation} \cite{DM_Perk+Au-Yang06}.
(\ref{DM_YB}) is a special case of (\ref{DM_Bruhat_R2}) for $n=2$,
\bez
    R_{ij} R_{ik} R_{jk} = R_{ijk} \; R_{jk} R_{ik} R_{ij}  \qquad \quad i<j<k \; . 
\eez
This generalization of (\ref{DM_YB}) has been considered in particular in  \cite{DM_Maillet+Nijhoff90,DM_Korepanov93,DM_Korepanov97,DM_Hiet+Nijh97}.

For example, in $\mathcal{B}_{4,1}$, we have
\bez
  (R_1 R_2)  R_3 R_4 &=& R_{12} \; R_2  (R_1 R_3)  R_4 
                        = R_{12} R_{13} \; (R_2 R_3)  R_1 R_4
                        = R_{12} R_{13} R_{23} \; R_3 R_2  (R_1 R_4) \\
                       &=& R_{12} R_{13} R_{23} R_{14} \; R_3  (R_2 R_4)  R_1
                        = R_{12} R_{13} R_{23} R_{14} R_{24} \; (R_3 R_4)  R_2 R_1 \\
                       &=& R_{12} R_{13} R_{23} R_{14} R_{24} R_{34} \; R_4 R_3 R_2 R_1 \; .
\eez 
The product with which we started corresponds to the minimal element of $B(4,1)$, i.e. 
$(\{1\},\{2\},\{3\},\{4\})$. 
After a sequence of inversions, leading to new elements of $B(4,1)$, the 
resulting product $R_4 R_3 R_2 R_1$ corresponds to the maximal element 
$(\{4\},\{3\},\{2\},\{1\})$. 
The final sequence of $R_{ij}$'s represents an element of $A(4,2)$. In $\mathcal{B}_{4,2}$, 
we have
\bez
 & & (R_{12} R_{13} R_{23}) R_{14} R_{24} R_{34}
  = R_{123} \; R_{23} R_{13}  (R_{12} R_{14} R_{24})  R_{34} \\
 &=& R_{123} R_{124} \; R_{23} [R_{13} R_{24}] R_{14} [R_{12} R_{34}] 
  = R_{123} R_{124} \; R_{23} R_{24}  (R_{13} R_{14} R_{34})  R_{12} \\
 &=& R_{123} R_{124} R_{134} \; (R_{23} R_{24} R_{34})  R_{14} R_{13} R_{12}
  = R_{123} R_{124} R_{134} R_{234} \; R_{34} R_{24} R_{23} R_{14} R_{13} R_{12} \, ,
\eez
where square brackets indicate an application of a commutativity relation, and 
\bez
    R_{12} R_{13} [R_{23} R_{14}] R_{24} R_{34}
  = \ldots 
  = R_{234} R_{134} R_{124} R_{123} \; R_{34} R_{24} R_{23} R_{14} R_{13} R_{12} \; .
\eez
These calculations reproduce by stepwise inversions the two maximal chains of $B(4,2)$, 
starting with the 
minimal element $(\{1,2\},\{1,3\},\{2,3\},\{1,4\},\{3,4\})$, and ending with the 
maximal element, which is $(\{3,4\},\{1,4\},\{2,3\},\{1,3\},\{1,2\})$ (cf. Figure~\ref{DM_fig:B42Qs}).
As a consequence of associativity, we obtain the consistency condition\index{consistency condition}
\bez
    R_{123} R_{124} R_{134} R_{234} = R_{234} R_{134} R_{124} R_{123} \; .
\eez
The two sides of this equation represent the two elements of $A(4,3)$. In $\mathcal{B}_{4,3}$, 
it generalizes to
\bez
    R_{123} R_{124} R_{134} R_{234} = R_{1234} \; R_{234} R_{134} R_{124} R_{123} \; .
\eez

We expect that, more generally, the consistency conditions\index{consistency condition} 
of (\ref{DM_Bruhat_R2}) are given by
\be
    R_{ L \setminus \{ l_{n+2} \} } \, R_{ L \setminus \{ l_{n+1} \} } \cdots 
    R_{ L \setminus \{ l_1 \} } 
  = R_{ L \setminus \{l_1\} } \, R_{ L \setminus \{l_2\} } \cdots 
    R_{ L \setminus \{l_{n+2}\} }    \label{DM_R-consistency}
\ee
for $L = \{l_1,\ldots,l_{n+2}\} \in {[N] \choose n+2}$ in linear order. 
These consistency conditions\index{consistency condition} have the form of generalized 
Yang-Baxter equations\index{Yang-Baxter equation}, the so-called 
\emph{simplex equations}\index{simplex equation}
\cite{DM_Zamolodchikov81,DM_Bazh+Stro82,DM_Frenkel+Moore91,DM_Lawrence95,DM_Lawrence97}, 
see the following remark. 
A derivation of these equations as consistency conditions\index{consistency condition}, 
in the way described above,  
apparently first appeared in \cite{DM_Maillet+Nijhoff90} (called `obstruction method' 
in \cite{DM_Michi+Nijh93,DM_Hiet+Nijh97}). 

\begin{remark}
\label{DM_rem:YB}
For given positive integers $n<N$, we choose two finite-dimensional 
vector spaces $V,W$, and $S_I \in  V \otimes \mathrm{End}(W)$, 
$I \in {[N] \choose n}$, with the property
\bez
    S_I \tilde{\otimes} S_{I'} = P \, S_{I'} \tilde{\otimes}  S_I  \qquad
    \mbox{if} \quad |I \cup I'| > n+1 \, ,
\eez
where $P(w \tilde{\otimes} w') = w' \tilde{\otimes} w$, and 
$\tilde{\otimes}$ denotes the tensor product over $\mathrm{End}(W)$. 
For $K = \{ k_1, \ldots, k_{n+1} \} \in {[N] \choose n+1}$, $k_1 < \ldots < k_{n+1}$, 
we write (\ref{DM_Bruhat_R2}) in the form
\bez
    S_{ K \setminus \{ k_{n+1} \} } \tilde{\otimes}  S_{ K \setminus \{ k_n \} } 
    \tilde{\otimes}  \cdots 
    \tilde{\otimes}  S_{ K \setminus \{ k_1\} } 
  = R \; S_{ K \setminus \{k_1 \} } \tilde{\otimes}   S_{ K \setminus \{k_2 \} } 
    \tilde{\otimes}  \cdots \tilde{\otimes}  S_{ K \setminus \{k_{n+1}\} } \, , 
\eez
with some $R \in \mathrm{End}(\otimes^{n+1} W)$. In components, this takes the form
\bez
   S^{a_{n+1}}_{ K \setminus \{ k_{n+1} \} } \, S^{a_n}_{ K \setminus \{ k_n \} } 
   \cdots S^{a_1}_{ K \setminus \{ k_1\} } 
  = \sum_{b_1,\ldots,b_{n+1}} {R\,}^{a_{n+1} \ldots a_1}_{b_{n+1} \ldots b_1} 
    \; S^{b_1}_{ K \setminus \{k_1 \} } \, S^{b_2}_{ K \setminus \{k_2 \} } \cdots 
    S^{a_{n+1}}_{ K \setminus \{k_{n+1}\} } \; .
\eez
The consistency conditions\index{consistency condition} are the 
$(n+1)$-simplex equations\index{simplex equation}
\bez
   R_{ L \setminus \{ l_{n+2} \} } \, R_{ L \setminus \{ l_{n+1} \} } 
   \cdots R_{ L \setminus \{ l_1\} } 
 = R_{ L \setminus \{l_1 \} } \, R_{ L \setminus \{l_2 \} } \cdots 
    R_{ L \setminus \{l_{n+2}\} } \, , 
\eez
for all $L = \{l_1,\ldots,l_{n+2}\} \in {[N] \choose n+2}$, 
$l_1 < l_2 < \cdots < l_{n+2}$. Here $R_K \in \mathrm{End}(\otimes^N V)$ is given by $R$ 
acting non-trivially only on factors of the $N$-fold tensor product of $V$ at those positions 
that are given by the numbers $k_1, \ldots, k_{n+1}$.
\end{remark}

\begin{remark}
\label{DM_rem:classical_simplex_eqs}
Assuming an extension of the monoid $\mathcal{B}_{N,n}$ to a unital ring of formal 
power series in an indeterminate $\epsilon$, and an expansion
\bez
    R_K = \mathbf{1} + \epsilon \, \mathfrak{r}_K + \cdots 
\eez
in powers of $\epsilon$, from (\ref{DM_R-consistency}) we obtain to first nontrivial 
order (which is $\epsilon^2$) 
\bez
    \sum_{1 \leq a < b \leq n+2} [ \mathfrak{r}_{L \setminus \{l_b\}} \, , \, 
        \mathfrak{r}_{L \setminus \{l_a\}} ]  = 0 \, ,
\eez
which has the form of a \emph{classical simplex equation} \cite{DM_Frenkel+Moore91}. 
For $N=3$, this is the \emph{classical Yang-Baxter equation} 
\bez
    [\mathfrak{r}_{12} , \mathfrak{r}_{13}] + [\mathfrak{r}_{12} , \mathfrak{r}_{23}] 
  + [\mathfrak{r}_{13} , \mathfrak{r}_{23}] = 0 \; . 
\eez
The `classical limit' does not work, however, on the level of the `obstruction equations' 
(\ref{DM_Bruhat_R2}). 
\end{remark}

\subsection{Equations associated with higher Tamari orders}
 Let $\mathcal{T}_{N,n}$ be the monoid obtained from $\mathcal{B}_{N,n}$ by 
replacing (\ref{DM_Bruhat_R2}) with the following \emph{$(n+1)$-gonal relations}, 
\be
  T_{ K \setminus \{ k_{n+1} \} } \, T_{ K \setminus \{ k_{n-1} \} } \cdots 
  T_{ K \setminus \{ k_{1+(n \, \mathrm{mod} \,2)} \} } 
  = T_K \;\, T_{ K \setminus \{k_{2 -(n \, \mathrm{mod} \,2)} \} } \cdots 
    T_{ K \setminus \{k_{n-2}\} } \, T_{ K \setminus \{k_n\} } \; .
         \label{DM_Tamari_(n+1)-gonal}
\ee
This means that we omit all factors in (\ref{DM_Bruhat_R2}) with a non-visible index set. 

Assuming that the corresponding statement for $\mathcal{B}_{N,n}$ holds, it 
follows that these relations imply the \emph{consistency conditions}\index{consistency condition}
\bez
  T_{ L \setminus \{ l_{n+2} \} } \, T_{ L \setminus \{ l_{n} \} } \cdots 
  T_{ L \setminus \{ l_{2+(n \, \mathrm{mod} \,2)} \} } 
  = T_{ L \setminus \{l_{1+(n \, \mathrm{mod} \,2)} \} } \cdots 
    T_{ L \setminus \{l_{n-1}\} } \, T_{ L \setminus \{l_{n+1}\} } 
\eez
for all $L = \{l_1,\ldots,l_{n+2}\} \in {[N] \choose n+2}$ in linear order. These are 
special $(n+2)$-gonal relations. For $n=1,2,3,4$, they are listed in the following 
table.
\bez
\begin{array}{c|@{\quad}c@{\quad}|@{\quad}c}
n & (n+1)-\mbox{gonal relations} & \mbox{consistency conditions} \\ 
\hline
1 & \Theta_i = X_{ij} \, \Theta_j & X_{ij} \, X_{jk} = X_{ik}  \\
2 & X_{ij} X_{jk} = Y_{ijk} \, X_{ik} & Y_{ijk} Y_{ikl} = Y_{jkl} Y_{ijl} \\
3 & Y_{ijk} Y_{ikl} = T_{ijkl} \, Y_{jkl} Y_{ijl} &  T_{ijkl} T_{ijlm} T_{jklm} = T_{iklm} T_{ijkm} \\
4 & T_{ijkl} T_{ijlm} T_{jklm} = S_{ijklm} \, T_{iklm} T_{ijkm} & 
     S_{ijklm} S_{ijkmq} S_{iklmq} = S_{jklmq} S_{ijlmq} S_{ijklq}
\end{array}
\eez
Here we demand that the indices are ordered such that e.g. $i<j$ for $n=1$ and $i<j<k<l<m<q$ for $n=4$, 
and we set
\bez
  \Theta_i = T_{\{i\}} \, , \quad
  X_{ij} = T_{\{i,j\}} \, , \quad
  Y_{ijk} = T_{\{i,j,k\}} \, , \quad
  T_{ijkl} = T_{\{i,j,k,l\}} \, , \quad
  S_{ijklm} = T_{\{i,j,k,l,m\}} \; .
\eez

 For example, for $n=3$, the tetragonal relations imply
\bez
    ( Y_{ijk} Y_{ikl} ) Y_{ilm} &=& T_{ijkl} Y_{jkl} Y_{ijl} Y_{ilm}
                                 = T_{ijkl} Y_{jkl} (Y_{ijl} Y_{ilm})
                                 = T_{ijkl} Y_{jkl} T_{ijlm} Y_{jlm} Y_{ijm} \\
                                &=& T_{ijkl} T_{ijlm} (Y_{jkl} Y_{jlm}) Y_{ijm} 
                                 = T_{ijkl} T_{ijlm} T_{jklm} Y_{klm} Y_{jkm} Y_{ijm} \\
  Y_{ijk} ( Y_{ikl} Y_{ilm} ) &=& Y_{ijk} T_{iklm} Y_{klm} Y_{ikm} 
                               = T_{iklm} Y_{ijk} Y_{klm} Y_{ikm}
                               = T_{iklm} Y_{klm} (Y_{ijk} Y_{ikm}) \\
                              &=& T_{iklm} Y_{klm} T_{ijkm} Y_{jkm} Y_{ijm} 
                               = T_{iklm} T_{ijkm} Y_{klm} Y_{jkm} Y_{ijm}   \, ,
\eez
where $i<j<k<l<m$, and we obtain the consistency conditions\index{consistency condition} 
for the $T$'s.

For $N=5$, the indices of $Y_{ijk}$ determine the three vertices of a triangle 
inside a pentagon, whose vertices are numbered consecutively counterclockwise
from $1$ to $5$. A product of three $Y$'s of the kind that appears in 
the above computations corresponds to a triangulation of the pentagon, 
see Figure~\ref{DM_fig:triang_pentagon}. The above computation starts with 
$Y_{ijk} Y_{ikl} Y_{ilm}$, i.e. the triangulation of the top vertex, and the computation 
proceeds step by step along the left, respectively right maximal chain. 

\begin{SCfigure}[1.6][hbtp]
\hspace{.5cm}
\includegraphics[scale=.5]{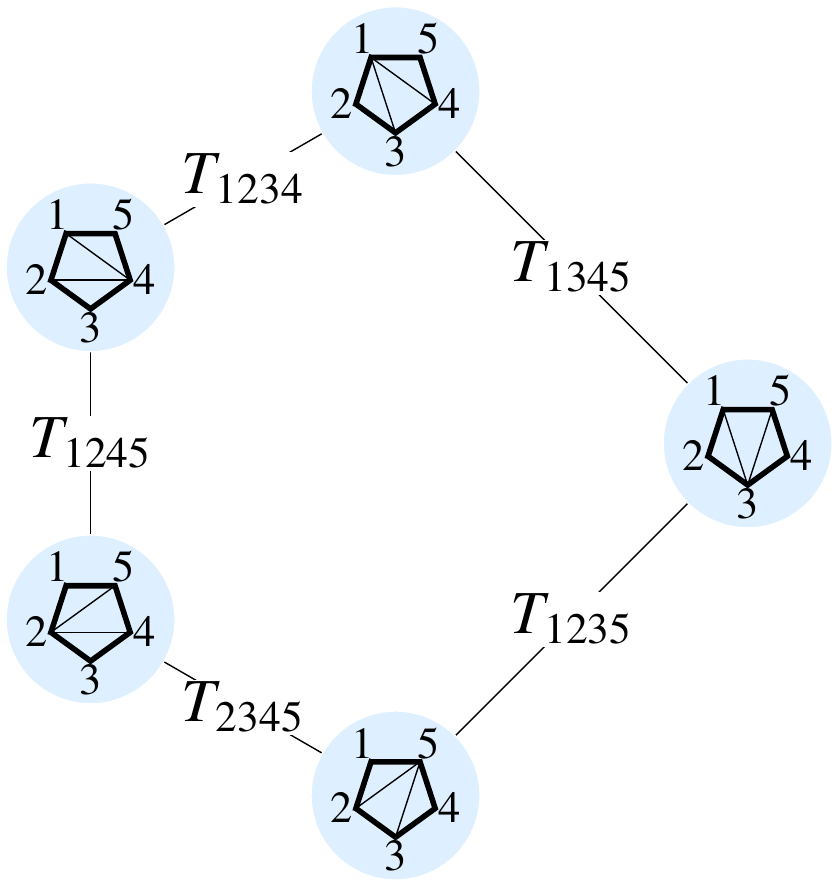} 
\hspace*{1.cm}
\caption{The two different ways of computing a product of three $Y$'s 
correspond to the maximal chains of the pentagonal lattice, which is 
the Tamari lattice\index{Tamari!lattice} $\mathbb{T}_3 = T(5,3)$. }
\label{DM_fig:triang_pentagon}  
\end{SCfigure} 

\vskip.1cm
A particularly interesting aspect is that we can construct any of the 
Tamari orders\index{order!higher Tamari} $T(N,n)$ 
by starting from $X_{12} X_{23} \cdots X_{N-1,N}$,
which corresponds to the minimal ele\-ment of the triangulations of the cyclic polytope 
$C(N,1)$ (see \cite{DM_Rambau+Reiner11}). Applying the trigonal relations, 
we obtain from it 
\bez
 (\cdots(X_{12} X_{23}) \cdots ) X_{N-2,N-1}) X_{N-1,N} 
 = Y_{123} Y_{134} Y_{145} \cdots Y_{1,N-1,N} \; X_{1,N} \; .
\eez
The chain of $Y$'s corresponds 
to the minimal element of the Tamari lattice $\mathbb{T}_{N-2} = T(N,3)$. Elaborating all 
possible proper bracketings of this chain, by applications of the tetragonal relations, one 
obtains the maximal chains of $T(N,3)$. For $N=6$, 
\bez
   (Y_{123} Y_{134}) Y_{145} Y_{156} = \ldots 
 = T_{1234} T_{1245} T_{2345} T_{1256} T_{2356} T_{3456} \; Y_{456} Y_{346} Y_{236} Y_{126}
     \, ,
\eez
produces the coefficient $T_{1234} T_{1245} T_{2345} T_{1256} T_{2356} T_{3456}$, 
which we recognize as one of the longest maximal chains of the 
Tamari lattice\index{Tamari!lattice} $\mathbb{T}_4$ in Figure~\ref{DM_fig:T4}. 
This computation moves step by step along that chain,  
from one tree to the next, since each product of four $Y$'s appearing in the computation 
determines the tree (which is dual to a triangulation of the $6$-gon) assigned to the respective vertex. 
In conclusion, we have an algebraic method to construct Tamari lattices and, more generally, 
higher Tamari orders\index{order!higher Tamari}. 
\vskip.1cm

Remark~\ref{DM_rem:YB} can be translated to the present setting. The significance of the 
equations obtained in this way has still to be explored, but an interesting observation is made 
in the following remark. 

\begin{remark}
Assuming an extension of the monoid $\mathcal{T}_{N,n}$ to a (unital) ring of formal 
power series in an indeterminate $\epsilon$, and an expansion
\bez
    T_I = \mathbf{1} + \epsilon \, \mathfrak{t}_I + \cdots \, ,
\eez
to first order in $\epsilon$, we obtain from (\ref{DM_Tamari_(n+1)-gonal}) the coboundary conditions
\bez
    \mathfrak{t}_K = (\delta \mathfrak{t})_K \quad \mbox{where} \quad
    (\delta \mathfrak{t})_K := \mathfrak{t}_{K\setminus \{k_{n+1}\}} - \mathfrak{t}_{K\setminus \{k_n\}} 
    + \mathfrak{t}_{K\setminus \{k_{n-1}\}} - \cdots + (-1)^n \, \mathfrak{t}_{K\setminus \{k_1\}} \; .
\eez
The consistency conditions are then a consequence of $\delta^2 =0$. In contrast to the `Bruhat case' 
considered in section~\ref{DM_ssec:Bruhat_eqs}, see Remark~\ref{DM_rem:classical_simplex_eqs}, 
in the present `Tamari case' all equations possess a (formal) `classical limit', which moreover 
turns out to be `cohomological'.
\end{remark}

\section{Further remarks}
\label{DM_sec:remarks}
The KP equation\index{KP!equation} (and its hierarchy)\index{KP!hierarchy} has connections 
with various areas of mathematics and (mathematical) physics. 
The relation with higher Bruhat\index{order!higher Bruhat} and higher 
Tamari orders\index{order!higher Tamari} established 
in \cite{DM_DMH11KPT} and in the present work provides another example.

We did not explore the precise relation of the higher Tamari orders\index{order!higher Tamari} 
introduced in this work and the `higher Stasheff-Tamari posets' of Kapranov and Voevodsky  \cite{DM_Kapr+Voev91} (also see \cite{DM_Edel+Rein96,DM_Rambau+Reiner11}). It may well be 
that these orders coincide, which is indeed so for low order examples. 

The higher Bruhat orders\index{order!higher Bruhat} and the induced Tamari 
orders\index{order!higher Tamari} are supplied with a higher-dimensional 
category structure via the evolution variables of the KP hierarchy\index{KP!hierarchy}. 
It has obvious relations to Ross Street's algebra of oriented simplexes \cite{DM_Street87} 
(also see \cite{DM_Kapr+Voev91}). 

\vskip.2cm
\noindent
\textbf{Acknowledgement}. 
We would like to thank Jim Stasheff for very valuable 
comments and questions that led to several improvements of this work, 
and for his efficient proofreading.

\providecommand{\bysame}{\leavevmode\hbox to3em{\hrulefill}\thinspace}
\providecommand{\MR}{\relax\ifhmode\unskip\space\fi MR }
\providecommand{\MRhref}[2]{%
  \href{http://www.ams.org/mathscinet-getitem?mr=#1}{#2}
}
\providecommand{\href}[2]{#2}

\end{document}